\documentclass{amsart} 
\usepackage{graphicx}
\usepackage{epsfig}

\usepackage{amssymb,amsmath,amsthm,amsfonts}

\usepackage[letterpaper, margin=1.2in]{geometry} 

\usepackage {latexsym,enumerate}
\usepackage{bbm}
\usepackage{xcolor}

\allowdisplaybreaks
\newcommand{\nc}{\newcommand}
\nc{\les}{\lesssim}
\nc{\nit}{\noindent}
\nc{\nn}{\nonumber}
\nc{\D}{\partial}
\nc{\diff}[2]{\frac{d #1}{d #2}}
\nc{\diffn}[3]{\frac{d^{#3} #1}{d {#2}^{#3}}}
\nc{\pdiff}[2]{\frac{\partial #1}{\partial #2}}
\nc{\pdiffn}[3]{\frac{\partial^{#3} #1}{\partial{#2}^{#3}}}
\nc{\abs}[1] {\lvert #1 \rvert}
\nc{\cAc}{{\cal A}_c}
\nc{\cE}{{\cal E}}
\nc{\cF}{{\cal F}}
\nc{\cP}{{\cal P}}
\nc{\cV}{{\cal V}}
\nc{\cQ}{{\cal Q}}
\nc{\cGin}{{\cal G}_{\rm in}}
\nc{\cGout}{{\cal G}_{\rm out}}
\nc{\cO}{{\cal O}}
\nc{\Lav}{{\cal L}_{\rm av}}
\nc{\cL}{{\cal L}}
\nc{\cB}{{\cal B}}
\nc{\cZ}{{\cal Z}}
\nc{\cR}{{\cal R}}
\nc{\cT}{{\cal T}}
\nc{\cY}{{\cal Y}}
\nc{\cX}{{\cal X}}
\nc{\cXT}{{{\cal X}(T)}}
\nc{\cBT}{{{\cal B}(T)}}
\nc{\vD}{{\vec \mathcal{D}}}
\nc{\efield}{\mathcal{E}}
\nc{\vE}{{\vec \efield}}
\nc{\vB}{{\vec \mathcal{B}}}
\nc{\vH}{{\vec \mathcal{H}}}
\nc{\mR}{\mathcal R}
\nc{\mF}{\mathcal F}
\nc{\ty}{{\tilde y}}
\nc{\tu}{{\tilde u}}
\nc{\tV}{{\tilde V}}
\nc{\Pc}{{\bf P_c}}
\nc{\bx}{{\bf x}}
\nc{\bX}{{\bf X}}
\nc{\bXYZ}{{\bf XYZ}}
\nc{\bY}{{\bf Y}}
\nc{\bF}{{\bf F}}
\nc{\bS}{{\bf S}}
\nc{\dV}{{\delta V}}
\nc{\dE}{{\delta E}}
\nc{\TT}{{\Theta}}
\nc{\dPsi}{{\delta\Psi}}
\nc{\order}{{\cal O}}
\nc{\Rout}{R_{\rm out}}
\nc{\eplus}{e_+}
\nc{\eminus}{e_-}
\nc{\epm}{e_\pm}
\nc{\sgn}{\text{sgn}}
\nc{\eps}{\varepsilon}
\nc{\vnabla}{{\vec\nabla}}
\nc{\G}{\Gamma}
\nc{\w}{\omega}
\nc{\mh}{h}
\nc{\mg}{g}
\nc{\vphi}{\varphi}
\nc{\tlambda}{\tilde\lambda}
\nc{\be}{\begin{equation}}
	\nc{\ee}{\end{equation}}
\nc{\ba}{\begin{eqnarray}}
	\nc{\ea}{\end{eqnarray}}
	\nc{\nrm}{\|}

\nc{\g}{\gamma}
\nc{\ol}{\overline}

\newtheorem{theorem}{Theorem}[section]
\newtheorem{lemma}[theorem]{Lemma}
\newtheorem{prop}[theorem]{Proposition}

\nc{\pT}{\partial_T}
\nc{\pz}{\partial_z}
\nc{\pt}{\partial_t}
\nc{\la}{\langle}
\nc{\ra}{\rangle}
\nc{\infint}{\int_{-\infty}^{\infty}}
\nc{\halfwidth}{6.5cm}
\nc{\figwidth}{10cm}
\newcommand{\f}{\frac}

\nc{\nlayers}{L} \nc{\nsectors}{M}
\nc{\indicator}{\mathbf{1}}
\nc{\Rhole}{R_{\rm hole}}
\nc{\Rring}{R_{\rm ring}}
\nc{\neff}{n_{\rm eff}}
\nc{\Frem}{F_{\rm rem}}
\nc{\R}{\mathbb R}
\nc{\mJ}{\mathcal J}
\nc{\C}{\mathbb C}
\nc{\Z}{\mathbb Z}
\nc{\N}{\mathbb N}
\nc{\DD}{\Delta}
\nc{\cD}{\mathcal D}
\nc{\lnorm}{\left\|}
\nc{\rnorm}{\right\|}
\nc{\rnormp}{\right\|_{\ell^{p,\eps}}}
\nc{\rar}{\rightarrow} 
\nc{\argz}{\text{arg}\,z}
\def\norm[#1][#2]{\|#1\|_{#2}}

\theoremstyle{remark}

\begin{document}
	
	\begin{abstract}
		
		We prove dispersive bounds for fractional Schr\"odinger operators on $\mathbb R^n$ of the form $H=(-\Delta)^{\alpha}+V$ with $V$ a real-valued, decaying potential and $\alpha \notin\mathbb N$.  We derive pointwise bounds on the resolvent operators for all $0<\alpha<\frac{n}{2}$, a quantitative limiting absorption principle for $\f12<\alpha<\frac{n}{2}$, and establish global dispersive estimates in dimension $n\geq 2$ for the range $\frac{n+1}{4}\leq \alpha <\frac{n}2$.  
		
	\end{abstract}

	\title[Dispersive estimates for fractional  Schr\"odinger operators]{\textit{Dispersive estimates for fractional order Schr\"odinger operators} } 
	
	\author[Erdo\smash{\u{g}}an, Goldberg, Green]{M. Burak Erdo\smash{\u{g}}an, Michael Goldberg and William~R. Green}
	\thanks{  The first author was partially supported by the NSF grant  DMS-2154031 and Simons Foundation Grant 634269.  The second author is partially supported by Simons Foundation
		Grant 635369. The third author is partially supported by Simons Foundation
		Grant 511825. }
	\address{Department of Mathematics \\
		University of Illinois \\
		Urbana, IL 61801, U.S.A.}
	\email{berdogan@illinois.edu}
	\address{Department of Mathematics\\
		University of Cincinnati \\
		Cincinnati, OH 45221 U.S.A.}
	\email{goldbeml@ucmail.uc.edu}
	\address{Department of Mathematics\\
		Rose-Hulman Institute of Technology \\
		Terre Haute, IN 47803, U.S.A.}
	\email{green@rose-hulman.edu}
	%\subjclass{35Q41, 42B20}

	\maketitle

	\section{Introduction}

The Schr\"odinger equation models how a quantum particle evolves in time.
A key feature of the free equation is dispersion: different frequency
components of the initial data (wave) propagate with different velocities,
so constructive interference from their interaction is always temporary. A wave packet built from a range of smilar frequencies will spread out and decrease in overall amplitude. On unbounded
domains such as $\mathbb{R}^n$, this causes the solution to decay in time
and to become smoother in an average sense. This decay estimate is not
just a fact about the linear equation; it is an important tool used to
control nonlinear equations.   For nonlinearities controlled by a monomial $|u|^{p-1}u$ with $p>1$,
the nonlinear terms
eventually become small relative to the linear terms because of this
decay.

Adding a background potential complicates the picture. The perturbed operator can
have a discrete spectrum, corresponding to bound states that oscillate in place without decay, and a
continuous spectrum, corresponding to states that scatter. The natural
question, studied since 
1970s, is whether the part of the evolution on the continuous spectrum
still satisfies the same dispersive bound as the free equation. Proving
this requires detailed estimates on the resolvent of the operator, both
near the bottom of the spectrum and at high energy, since the dispersive
bound is obtained from an oscillatory integral built from the resolvent.
This program is well understood for the classical (second order)
Schr\"odinger operator $-\Delta +V$, and has more recently been extended to higher
order operators such as $(-\Delta)^m + V$ for integer $m$.

This paper studies the case where the order is a non-integer real number
$\alpha$. Such fractional order operators arise in several settings: as
continuum limits of long-range lattice interactions, in models of
optical systems with fractional dispersion, and in Laskin's theory of
fractional quantum mechanics. The extension from integer to non-integer
order is not routine. When $m$ is an integer, there is an explicit
algebraic identity that writes the $m$th order resolvent as a finite sum
of classical Schr\"odinger resolvents, and the existing higher order
theory relies heavily on this identity. No such finite identity exists
for non-integer $\alpha$, so the resolvent must be analyzed directly. It turns out to have features with no integer order analogue: its low
energy expansion is more intricate than in the integer order case,
since it involves two power series on different scales, one in integer
powers of $\lambda$ and one in powers of $\lambda^{1/\alpha}$. The main results of this paper establish the expected
$L^1 \to L^\infty$ dispersive bound for fractional order Schr\"odinger
operators in a natural range of $\alpha$ and $n$, together with resolvent
expansions that should be of independent use.

We now turn to a more precise account of the problem. We  study properties of solutions to the family of Schr\"odinger equations on $(x,t)\in \mathbb R^{n}\times \mathbb R$:
	\begin{align}\label{eqn:schro pde}
		iu_t+H_\alpha u=0, \qquad H_\alpha:=(-\Delta)^\alpha+V,
	\end{align}
	where $V(x)$ is a bounded and decaying potential and $0<\alpha<\frac{n}2$ need not be an integer. When $\alpha=1$, this is the classical Schr\"odinger operator that arises in quantum mechanics to describe how the wave function of a quantum particle evolves in the presence of an external potential. For general $\alpha$, the fractional Laplacian $(-\Delta)^{\alpha}$ is defined as a Fourier multiplier operator with symbol $|\xi|^{2\alpha}$.  It behaves as one might expect with regard to the smoothness of functions but has non-local action when $\alpha$ is not an integer. For smooth compactly supported functions $f$ and an integer $m > \alpha$, there is an explicit representation
\[
(-\Delta)^\alpha f(x) = c_{m,\alpha} (-\Delta)^m \int_{\R^n} |x-y|^{2(m-\alpha) - n} f(y) dy.
\]
where the constant $c_{m,\alpha}$ is a ratio of certain Gamma function values.

Most dynamical behavior of solutions of \eqref{eqn:schro pde} can be traced back to properties of the resolvent of $H$, which itself is a perturbation of the resolvent of $(-\Delta)^\alpha$. We derive asymptotic descriptions of the resolvent of $(-\Delta)^\alpha$ in both low and high energy regimes. From these we obtain dispersive $L^1 \to L^\infty$ bounds for solutions of \eqref{eqn:schro pde} which are orthogonal to the finite dimensional space of bound states. In addition, the resolvent expansions we obtain are useful in the study of $L^p$-boundedness of the wave operators, which we establish in a companion paper \cite{EGGfracwave}.

	 Dispersive equations of this type arise in modeling diverse wave phenomena in physics and applied mathematics, including applications to  {\it nonlinear optics, water wave theory, anharmonic lattices, elastic rods, plasma physics and macroscopic ferromagnetic continua},   \cite{SulemSulem}.  In particular, fractional Schr\"odinger  (FS) equations have attracted attention in the physics literature, see for example \cite{Laskin,Laskin2}, where Laskin proposes a theory of fractional Quantum Mechanics.
	Posed on the real line, the fractional cubic non-linear Schr\"odinger (NLS) equation has appeared   in many recent articles.  A rigorous derivation of the equation can be found in \cite{kay} starting from a family of models describing charge transport in
bio polymers like  DNA. The starting point is a discrete NLS
  equation with general lattice interactions. A fractional cubic nonlinear Schr\"odinger equation with $\alpha\in(\frac12,1)$ appears as the continuum limit of the long-range interactions between quantum particles on the lattice. Whereas  allowing only the short-range interactions the authors obtain the standard cubic NLS. 
  In optics, the FS equation was first proposed by Longhi in 2015 for transverse light dynamics in optical cavities \cite{Longhi}, with the first experimental realization achieved in 2023 using programmable holograms to emulate fractional group-velocity dispersion in fiber optics \cite{LZMK}. These systems support various soliton families including gap solitons \cite{HD}, vortex solitons  \cite{XTHD}, and self-bound modes with competing cubic-quintic nonlinearities \cite{LMM}, where spontaneous symmetry breaking bifurcations exhibit unique features depending on the fractional power of the Laplacian.  Finally, see \cite{KC} for a collection of articles on the properties and applications of fractional dispersive equations. 

Recall that the solution of  the  free Schr\"{o}dinger equation ($\alpha=1, V=0$) on $\mathbb{R}^n$,
$$ e^{-it\Delta}f(x) = C_n \frac{1}{t^{n/2}}\int_{\mathbb{R}^n} e^{-i|x-y|^2/4t}f(y)\, dy,$$
satisfies three basic dispersive estimates:
\begin{enumerate}[i)]
\item $L^1\to L^\infty$ estimates:
$$\|e^{-it\Delta}f\|_{L^\infty(\mathbb{R}^n)} \lesssim t^{-n/2} \|f\|_{L^1(\mathbb{R}^n)},$$
\item Strichartz estimates:
$$\|e^{-it\Delta}f\|_{L_t^q(L_x^p)}\lesssim \|f\|_{L^2(\mathbb{R}^n)}, \quad \frac{2}{q}+\frac{n}{p}=\frac{n}{2}, \quad 2\leq p,q \leq \infty,\quad (n,q)\neq (2,2),$$
\item Kato smoothing estimates: For $\beta>\frac12$,
$$\big\|(1+|x|)^{-\beta} |\nabla|^{1/2}e^{-it\Delta} f\big\|_{L^2_{x,t}} \lesssim \|f\|_{L^2(\mathbb{R}^n)}.$$
\end{enumerate}
The three estimates are closely interrelated. While the $L^1\to L^\infty$ bound limits the amplitude of a wave function at one snapshot in time, Strichartz inequalities give a holistic bound of how large the amplitude can be over what volume of space, for how long.  Kato smoothing occurs because the higher-frequency (i.e. less smooth) part of a  solution moves with faster group velocity, hence it spends less time in any given bounded set.
 
Strichartz estimates were first obtained by Strichartz \cite{strichartz} via Fourier restriction theory of Stein. They can also be obtained using $L^1\to L^\infty$ estimates (see \cite{GVs}, \cite{keetao}), or (for perturbed equations) using Kato smoothing estimates \cite{RodSch}.
	
		Each of these bounds extend, with some modification, to the perturbed case when $V\not\equiv 0$ given appropriate restrictions on the potential.  We focus on the $L^1\to L^\infty$ bounds.  To illustrate we consider the perturbed Schr\"odinger operator $H_1=-\Delta+V $  on $\R^n$, where $V$ is a real-valued   potential satisfying $|V(x)|\leq C (1+|x|)^{-\beta}$ for some $\beta > 1$.  In this case, the spectrum of $H_1$ is composed of the absolutely continuous spectrum $[0,\infty)$ and finitely many non-positive eigenvalues (bound states), \cite{RS1}.  
		With $\lambda_j$ an eigenvalue of $H_1$, $P_j$ a projection onto the eigenspace associated to $\lambda_j$, and $P_{ac}$ the projection onto the absolutely continuous spectrum, we have
	\begin{align}\label{eqn:proj espace}
		e^{itH_1}f(x)=\sum_{j=1}^N e^{it\lambda_j}P_jf(x)+e^{itH_1}P_{ac}(H_1)f(x).
	\end{align}
	Since $|e^{it\lambda_j}| = 1$ does not decay at $t \to \infty$, the evolution $e^{itH_1}$ can not satisfy any of the dispersive estimates listed above unless the number of bound states, $N$, happens to be zero. We now turn to the remaining component involving the projection onto the absolutely continuous spectrum. For bounded, sufficiently decaying and smooth  real-valued potentials $V$, the time decay of $e^{itH_1}P_{ac}(H_1)$ is well studied. In general, it satisfies the same dispersive bound as the free propagator:
\be\label{eq:global} \| e^{itH_1}P_{ac}(H_1) \|_{L^1 \to L^\infty} \lesssim |t|^{-\frac{n}{2}}.
\ee
Interpolating this bound with the $L^2$ conservation law yields a family of $L^p \to L^{p'}$ estimates. These, combined with a $T^*T$ argument, lead to Strichartz estimates for the linear propagator as mentioned above. Such estimates are crucial in analyzing the stability of soliton or standing wave solutions to nonlinear Schr\"odinger equations.

 It is common to consider $e^{itH_1}P_{ac}(H_1)$ as an element of the functional calculus of the operator $H_1$. By Stone's formula, we have  
\begin{align}\label{eqn:Pac H1}
    e^{itH_1}P_{\mathrm{ac}}(H_1) = \int_{0}^\infty e^{it\lambda} \, d\mu_1(\lambda).
\end{align} 
The spectral measure, $\mu_1$,  can be constructed in terms of the resolvent operators. For any $z\in \C\setminus \R$,   
  we define the resolvent operator by
$$ R_V(z) = (H_1 - zI)^{-1},$$ 
where $I$ denotes the identity operator.
One studies the perturbed resolvent operators $R_V$  as perturbations of the free resolvent operator,  whose convolution kernel is given explicitly in terms of Hankel functions:
	$$
		R_0(z)(x,y)=\frac{i}{4}\bigg( \frac{z^{\f12}}{2\pi |x-y|} \bigg)^{\frac{n}2-1} H_{\frac{n}2}^{(1)} (z^{\f12}|x-y|).
	$$
	This allows for closed form expansions when $n$ is odd, while when $n$ is even one has logarithmic singularities that complicate the process.
	
  For $\lambda \in [0,\infty)$, we define the limiting resolvents by
$$R_V^\pm(\lambda) = \lim_{\epsilon \to 0^+} (H_1 - (\lambda \pm i\epsilon)I)^{-1}.
$$
These limits exist as operators between certain weighted $L^2$ spaces by the limiting absorption principle of Agmon \cite{agmon}. 
Each limiting resolvent possesses an integral kernel so its action may also be understood by 
$$
R_V^\pm(\lambda) f(x)=\int_{\R^n} R_V^\pm(\lambda)(x,y) f(y)\, dy.
$$
The spectral measure $\mu_1$   can be written in terms of the limiting resolvents as
$$ d\mu_1(\lambda)(x,y) = \frac{1}{2\pi i} \big( R_V^+(\lambda)(x,y) - R_V^-(\lambda)(x,y) \big) \, d\lambda.
$$
The linear propagator $e^{itH_1}P_{\mathrm{ac}}(H_1)(x,y)$ is then recovered from the Stone formula \eqref{eqn:Pac H1}.
For the $L^1\to L^\infty$ bounds it suffices to prove the following pointwise bound on the integral kernel
$$
\sup_{x,y\in \mathbb R^n} \bigg|\frac{1}{2\pi i} \int_0^\infty e^{it\lambda} \big( R_V^+(\lambda) - R_V^-(\lambda) \big)(x,y)\, d\lambda\bigg|\leq C_n |t|^{-\frac{n}{2}}.
$$

Typically, dispersive decay estimates  are proven by developing detailed expansions of the resolvent and its derivatives near the threshold $\lambda = 0$ (low energy) and for large $\lambda$ (high energy).  It is crucial to have 	detailed information on the resolvent operators and their derivatives, while also capturing cancellation in the difference between the limiting resolvents near the threshold, $\lambda=0$. 
 These expansions reduce the problem to oscillatory integral estimates. We assume that the threshold energy is neither an eigenvalue nor a resonance, to avoid a finite-rank component of the evolution with slower decay. The high-energy analysis generally requires some smoothness of the potential in higher dimensions $n \geq 4$. 	In addition,  the fact that the classical Schr\"odinger operator has no  positive eigenvalues is of great importance.
These types of bounds were first investigated by Journe, Soffer, and Sogge \cite{JSS}, and have their roots in time-decay estimates between weighted $L^2$ spaces due to Rauch, Jensen, Kato, and Murata \cite{Rau,JenKat,Jen,Mur}. For a comprehensive history and further references, see the survey by Schlag \cite{SchlagSurvey}.

	Recent work, for example \cite{fsy,GT4,egt,soffernew,EGGHighOrder,LSY,cheng1,cheng2}, has extended the analysis from the classical $\alpha=1$ Schr\"odinger operator to higher integer order operators $(-\Delta)^m+V$ where $m\in \mathbb N$.  Such operators provide a model of ``high dispersion" phenomena and have applications in optical phenomena such as the propagation of laser beams, \cite{K,KS}.   The analysis of these operators relies on   the spectral method described above; developing resolvent expansions and reducing to oscillatory integrals.
	
	In this case, one utilizes the splitting identity which relates the integer order resolvent operators to the classical, $\alpha=1$, Schr\"odinger resolvents (c.f. \cite{soffernew}):   for $z\in\C\setminus[0,\infty)$, 
	\be\label{eqn:Resol}
	\mR_0(z)(x,y):=((-\Delta)^m -z)^{-1}(x,y)=\frac{1}{ mz^{1-\frac1m} }
	\sum_{\ell=0}^{m-1} \omega_\ell R_0 ( \omega_\ell z^{\frac1m})(x,y)
	\ee
	where $\omega_\ell=\exp(i2\pi \ell/m)$ are the $m^{th}$ roots of unity, $R_0(z)=(-\Delta-z)^{-1}$ is the usual ($2^{nd}$ order) Schr\"odinger resolvent.   
	This follows by applying the algebraic identity
	$$
		\frac{1}{a^{2m}-b^{2m}}=\frac{1}{\prod_{\ell=0}^{m-1} (a^2-\omega_\ell b^2)} 
		%=\frac{1}{b^{2m}}\sum_{\ell=0}^{m-1} \bigg( \frac{\frac{1}{m}\omega_\ell}{(a/b)^2-\omega_\ell} \bigg)
		=\frac{1}{mb^{2m-2}}\sum_{\ell=0}^{m-1}\frac{\omega_\ell}{a^2-\omega_\ell b^2}
	$$
	on the Fourier side. 
	
	This identity allows one to develop appropriate expansions for the resolvent operators by leveraging the more well-known $m=1$ resolvent operators.  Bounds on the derivatives of the resolvent operators and any cancellation between the limiting operators is vital in the low energy analysis while decay in the spectral parameter $\lambda$ is needed in the analysis of the high energy.  As in the classical case, one can reduce to an oscillatory integral with respect to the spectral measure corresponding to the difference of the limiting perturbed resolvent operators and obtain a natural decay rate of $|t|^{-n/(2m)}$,  see \cite{fsy,soffernew,GT4,egt,cheng1,cheng2} for example.
	In the higher integer order case, one must add the assumption that $H_m$ has no positive eigenvalues.  It can be shown that, even for a compactly supported, smooth potential $V$, that $H_m$ can have positive eigenvalues, while for $H_1$ sufficient decay on the potential rules out positive eigenvalues.

	When considering the non-integer order case, there is no hope of developing a finite term analogue of the splitting identity, \eqref{eqn:Resol}, that is crucial for the analysis of the higher integer order analysis.  This, along with the non-local nature of the fractional order operators, provides significant technical challenges.  Our main technical achievement is to develop detailed expansions on the resolvent operators, and their derivatives, in the non-integer order case in Section~\ref{sec:res} below.  From these expansions, we deduce properties of the spectral measure $d\mu_\alpha$ associated to an operator $H_\alpha$ and time decay estimates that generalize \eqref{eq:global}.

	\subsection{Main Results}
	We now consider the non-integer order Schr\"odinger operators $H=(-\Delta)^\alpha +V$, and no longer include a subscript.
	We assume that the potential does not introduce any embedded eigenvalues of positive energy, and that zero is a regular point of the spectrum for both $H$ and $(-\Delta)^\alpha$ (the last assumption being satisfied precisely when $2\alpha<n$).  Under these assumptions we show that the perturbed evolution $e^{itH}$ satisfies the same $L^1 \to L^\infty$ bounds as that of the free fractional Laplacian $e^{it(-\Delta)^\alpha}$ once any bound states are projected away.  

	The free propagator kernel is the Fourier transform in $\R^n$ of the complex exponential function $e^{it|\xi|^{2\alpha}}$.  By a scaling argument, it should have size $|t|^{-n/(2\alpha)}$ provided the Fourier transform of $e^{i|\xi|^{2\alpha}}$ is bounded. A stationary phase estimate shows that this occurs whenever $\alpha \geq 1$.  Thus when $\alpha \geq 1$ we seek a bound of the form
	$$
		\|e^{itH} P_{ac}(H)\|_{L^1\to L^\infty} \les |t|^{-\frac{n}{2\alpha}}.
	$$
	Here $P_{ac}$ is projection onto the absolutely continuous spectrum of $H$.

	For $\frac12 < \alpha < 1$, the convolution kernel of $e^{i(-\Delta)^\alpha}$ grows along with $|x-y|$ and oscillates, so smoothing of some order is needed in order to obtain a uniform bound.  More specifically, the Fourier transform of $e^{i|\xi|^{2\alpha}}|\xi|^{\gamma - n}$ is bounded  when $\frac12 < \alpha < 1$ and $0  < \gamma \leq n\alpha$.  The minimum amount of homogeneous smoothing required is of order $n(1-\alpha)$, and in that case the scaling considerations yield that $e^{it(-\Delta)^{\alpha}}(-\Delta)^{n(\alpha - 1)/2}$ satisfies an $L^1 \to L^\infty$ bound with size $|t|^{-\frac{n}{2}}$.  In the cases where $\frac12 < \alpha < 1$ we seek a bound of the form
	$$
		\|e^{itH}H^{\frac{n}{2}(1-\frac1\alpha)} P_{ac}(H)\|_{L^1\to L^\infty} \les |t|^{-\frac{n}{2}}.
	$$
Smoothing to higher order (i.e. choosing $\gamma < n\alpha$) results in a slower rate of time decay of the linear propagator.

The range of $\alpha$ and $n$ in our argument is governed by two considerations.  The assumption above that zero is a regular point of the spectrum of $(-\Delta)^\alpha$ is true when $0 < 2\alpha < n$ and false otherwise.  We also rely on a uniform bound for the kernel of the free resolvents $(-\Delta)^\alpha - (\lambda + i0))^{-1}$ for large $\lambda$. This bound will only be true when $\alpha \geq \frac{n+1}{4}$.  Consequently our results are stated for fractional Laplacian operators in the range $\frac{n+1}{4} \leq \alpha < \frac{n}{2}$.  The range is empty in one dimension, consists of the interval $[\frac34, 1)$ in two dimensions, and is contained in the half-line $\alpha \geq 1$ in all dimensions $n \geq 3$.

	Our main result(s) are
	
	\begin{theorem}\label{thm:main2}
		
		In dimension $n=2$, fix $\frac34\leq\alpha<1$.  Assume that $|V(x)|\les \la x\ra^{-\beta}$ for some $\beta>4$ and that $H$ has no embedded eigenvalues and zero is a regular point of the spectrum.  Then
		$$
			\|e^{itH}H^{1-\frac1\alpha} P_{ac}(H)\|_{L^1\to L^\infty} \les |t|^{-1}.
		$$
		
	\end{theorem}
	
	The two-dimensional result requires a certain amount of smoothing as discussed above.  In dimensions $n\geq 3$, we prove global bounds of the form

	\begin{theorem}\label{thm:main n}
		
		In dimensions $n\geq 3$, fix $\frac{n+1}4\leq\alpha<\frac{n}2$.  Assume that $|V(x)|\les \la x\ra^{-\beta}$ for some $\beta>n+4$ and that $H$ has no embedded eigenvalues and zero is a regular point of the spectrum.  Then
		$$
		\|e^{itH} P_{ac}(H)\|_{L^1\to L^\infty} \les |t|^{-\frac{n}{2\alpha}}.
		$$
		
	\end{theorem}

	These results have been established before in the cases where $\alpha$ is an integer.  The case $\alpha = 1$, $n=3$ is particularly well known, with Theorem~\ref{thm:main n} being true for $\beta > 2$,~\cite{GoldRough} and scaling-critical conditions that approximate $\beta = 2$,~\cite{BG}.  Theorem~\ref{thm:main n} is true in cases with integer $\alpha \geq 2$ and scaling-critical conditions that include all $\beta > 2\alpha$,~\cite{EGGHighOrder}.

	We note that when $\alpha\notin \mathbb N$ that $(-\Delta)^{\alpha}$ is a non-local operator unlike the integer order operators.  However, it is still possible to a certain extent to use the integer-$\alpha$ cases as a guide to what occurs in general. Indeed, one might claim that Theorems~\ref{thm:main2} and~\ref{thm:main n} hold because such a heuristic is valid for a sufficient number of the leading order terms.    It is important to emphasize here that the resolvent of the fractional Laplacian displays some behavior that is completely absent when $\alpha$ is an integer, causing the heuristics to break down after a finite number of terms which is fortunately large enough to permit the dispersive estimates to pass through unscathed.  At high energy the culprit is the singularity at zero of the function $|\xi|^{2\alpha}$.  At low energy the asymptotic expansion of the free resolvent becomes a concatenation of two distinct power series, one in integer powers of $\lambda$ and one in powers of $\lambda^{1/\alpha}$, with logarithmic corrections if the two happen to coincide.  These expressions are more intricate than what occurs for $\alpha\in \mathbb N$, and additional care is required.
	
	To the best of our knowledge, there are no known results on global dispersive bounds for perturbed fractional Schr\"odinger operators.
	Cho, Ozawa, and Xia studied dispersive and Strichartz estimates for the free operator assuming initial data in distorted Besov spaces, \cite{COX}.  Further study of Strichartz estimates for related operators may be found in \cite{CL,GW}, for example.  Recently, Taira considered time decay estimates between weighted $L^2$ spaces for positive potentials, \cite{Taira}.  We consider related problems with the $L^p$ boundedness of the wave operators in a companion paper, \cite{EGGfracwave}.
	
	The lower bound on $\alpha$ arises due to growth of the resolvent in the spectral parameter of order $\lambda^{\frac{n+1}{2}-2\alpha}$ that we derive in Proposition~\ref{prop:R0 exp} below.  These uniform bounds may be of independent interest.  As in the integer order case, \cite{GV,EGGcounter}, we believe that some smoothness of the potential is required in general for dispersive bounds to hold if $n>4\alpha-1$.  The upper bound on $\alpha$ arises to avoid the existence of zero energy resonances for the free operator $(-\Delta)^{\alpha}$.  For a more thorough discussion of the history of dispersive estimates in the integer order case we refer to \cite{EGGHighOrder}.

	We note that one can apply standard arguments to deduce families of Strichartz estimates from the dispersive bounds in Theorems~\ref{thm:main2} and \ref{thm:main n}.
	
	The paper is organized as follows.  In Section~\ref{sec:res} we develop detailed expansions of the resolvent operators and prove a quantitative limiting absorption principle for the fractional Schr\"odinger operators.  Then in Section~\ref{sec:2d} we prove Theorem~\ref{thm:main2}, and in Section~\ref{sec:nd} we prove Theorem~\ref{thm:main n}.  Finally, in Section~\ref{sec:spec} we provide a characterization of the regularity of the threshold.
	
	\section{Resolvent estimates}\label{sec:res}
	
	For short-range potentials, such as those satisfying $|V(x)|\les \la x\ra^{-1-}$, we refer the reader to \cite{ZHZ} for a limiting absorption principle with uniform bounds on compact subsets of $(0,\infty)$ under the assumption that there are no embedded eigenvalues.  In particular, we have boundedness of the resolvents from $L^{2,\f12+}$ to $L^{2,-\f12-}$. 
	
	However we need more detailed information on the perturbed and free resolvents to study the dispersive estimates for the evolution.
	In this section we establish pointwise bounds on the limiting free resolvent operators, 
	$$
		\mR_0^\pm(\lambda)=\lim_{\epsilon\to 0^+} [(-\Delta)^{\alpha}-(\lambda\pm i\epsilon)]^{-1},
	$$
	and their derivatives.  These in turn we may use to understand the perturbed resolvent operators
	$$
		\mR_V^\pm(\lambda)=\lim_{\epsilon\to 0^+} [(-\Delta)^{\alpha}+V-(\lambda\pm i\epsilon)]^{-1}.
	$$
 	Specifically, we show
	
	\begin{prop}\label{prop:rep}

		Fix $\alpha>0$ and $n\in \mathbb N$ with $n>2\alpha$. Then,   we have the representations, with $r=|x-y|$, 
		$$
		\mR_0^+(\lambda^{2\alpha})(x,y)
		=  \frac{e^{i\lambda r}}{r^{n-2\alpha}} F(\lambda r ), \text{ and}$$ 
		\be\label{Frep}
		[\mR_0^+(\lambda^{2\alpha})-\mR_0^-(\lambda^{2\alpha})](x,y)= \lambda^{n-2\alpha}  \big[ e^{i\lambda r}F_+(\lambda r)+e^{-i\lambda r}F_-(  \lambda r)\big],
		\ee
		where, for all $0\leq N\leq \frac{n+1+4\alpha}2$, 
		\be\label{Fbounds}
		|\partial_\lambda^N F(\lambda r)|\les \lambda^{-N} \la  \lambda r \ra^{\frac{n+1}2 -2\alpha},\, \,\,\, \,\,\, 
		|\partial_\lambda^N F_\pm(\lambda r)|\les \lambda^{-N} \la  \lambda r\ra^{-\frac{n-1}2}.
		\ee 
		Further,  
		for all $1\leq N\leq \frac{n+1+4\alpha}2$ we have
		\be\label{Fbounds2}
		|\partial_\lambda^N F(\lambda r)|\les\lambda^{-N}(\lambda r)^{\min(1,n-2\alpha,2\alpha -)} \la  \lambda r \ra^{\frac{n+1}2 -2\alpha},\quad 
		|\partial_\lambda^N F_{\pm}(\lambda r)|\les\lambda^{-N}(\lambda r) \la  \lambda r \ra^{-\frac{n-1}2},
		\ee 
		which improves the estimate above for $\lambda r\les 1$.
		
	\end{prop}

	For the low energy argument, we define $\log^-(y)=-\log(y)\chi_{\{0<y<1\}}$, and use the following expansions.

	\begin{prop}\label{prop:R0 exp}
		Fix $\alpha>0$ and $n\in \mathbb N$ with $n>2\alpha$.
		For $0<\lambda <1$,
		$$
			\mR_0^+(\lambda^{2\alpha})(x,y)
			=  \frac{C_\alpha}{|x-y|^{n-2\alpha}} + \mathcal E(\lambda,r)
		$$
		where
		$\mathcal E(\lambda, r)=O(\lambda^{n-2\alpha})$ when $4\alpha>n$, $\mathcal E(\lambda, r)=O(\lambda^{n-2\alpha}(1+\log^-(\lambda r)))$ when $n=4\alpha$ and 
		$\mathcal E(\lambda, r)=O(\lambda^{n-2\alpha}+\lambda^{2\alpha}r^{4\alpha-n})$ when $n>4\alpha$.
		
	\end{prop}

	Finally, we establish a limiting absorption principle for large energies.
	\begin{prop}\label{prop:LAP}
		
		Fix $\alpha>\f12$ and $n>2\alpha$.  Assume that $H$ has no embedded eigenvalues.  Then when $\lambda\gtrsim 1$, we have
		$$
			\|\la x\ra^{-\f12-}\mR_V^\pm(\lambda^{2\alpha}) \la y\ra^{-\f12 -} \|_{L^2\to L^2} \les  \lambda ^{1-2\alpha},
		$$
		provided that $|V(x)|\les \la x\ra^{-\beta}$ for some $\beta>1$.  Further, for any $j\in \N$, if $\beta>1+2j$, we have
		$$
			\|\la x\ra^{-j-\f12-}\partial_\lambda^j\mR_V^\pm(\lambda^{2\alpha}) \la y\ra^{-j-\f12 -} \|_{L^2\to L^2} \les  \lambda ^{1-2\alpha}.
		$$
	 
	\end{prop}

To prove these representations we use the following bounds.
\begin{lemma}\label{lem:small xi}
	
	Let $g$ be compactly supported on $\R^n$, and smooth on $\mathbb R^n\setminus\{0\}$, with $|\nabla^k g(\xi)|\les |\xi|^{\gamma-k}$ for some $\gamma>-n$ and $k=0,1,\dots$ for $\xi\neq 0$. Then
	$|\widehat g(x)|\les \la x\ra^{-n-\gamma}$.  In particular, $\widehat g\in L^1$ if $\gamma>0$.
	
	Furthermore, for $j\geq 1$ we have $|\nabla^j \widehat{g}(x)|\les \la x\ra^{-n-j-\gamma}$.
	
\end{lemma}
\begin{proof}
	The bound is clear for $|x|\les 1$.  For $|x|\gtrsim 1$, we write
	$$
	\widehat g(x)=\int_{\R^n} e^{-2\pi i x\cdot \xi}g(\xi)[\chi(\xi |x|)+\widetilde \chi(\xi |x|)]\, d\xi
	$$
	When $|\xi| |x|\les 1$, the bound is clear by converting to polar coordinates. For $|\xi| |x|\gtrsim 1$, we integrate by parts $k>\gamma +n$ times to bound by
	$$
	\int_{|\xi|\gtrsim |x|^{-1} } |\xi|^{\gamma-k} |x|^{-k}  \,d\xi \les |x|^{-\gamma-n}.
	$$
	The claim for derivatives follows because $|\xi|^j g$ satisfies the hypotheses with $\gamma+j$ in place of $\gamma$.

\end{proof}

\begin{lemma}\label{lem:FT}
	
	Let $g$ be a smooth function, supported  away from zero on $\mathbb R^n$, that satisfies $|\nabla^k g(\xi)|\les |\xi|^{\gamma-k}$ for some $\gamma<0$ and $k=0,1,2,\dots$. Then $\widehat{g}$ is a smooth function on $\mathbb \R^n\setminus\{0\}$ satisfying 
	$$
		|\nabla^N\widehat{g}(x)|\les  \left\{
		\begin{array}{ll}
			|x|^{-\gamma-n-N} & \text{if }\gamma+N>-n,\\
			|\log |x|\,| & \text{if }\gamma+N=-n,\\
			1 & \text{if }\gamma+N<-n.
		\end{array}
		\right. 
	$$
	Morever for $|x|\gtrsim 1$, $|\nabla^N\widehat{g}(x)|\les |x|^{-M}$ for all $M,N$.

\end{lemma}

\begin{proof}
	
	Noting that $\nabla^k   g\in L^1$ for sufficiently large $k$, we have, up to a distribution $u$ supported at zero, $\widehat g$ is a continuous function satisfying $|\widehat{g}(x)|\les |x|^{-M}$ for all $M$ and $x\neq 0$.  Since $g$ decays at infinity, $u=0$.  Similarly, $\nabla^k |\xi|^N  g\in L^1$ for sufficiently large $k$, so the derivatives also decay rapidly at infinity.
	
	To obtain the bounds for small $|x|\les 1$, we repeat the argument of Lemma~\ref{lem:small xi} above for $|\xi|^Ng(\xi)$, bounding the integral $1\les |\xi|\les |x|^{-1}$ directly and integrating by parts sufficiently many times when $|x||\xi|\gtrsim 1$.
	
\end{proof}

\begin{proof}[Proof of Proposition~\ref{prop:rep}]

For a complex number $z$ with $0 < \argz < \frac{\pi}{\alpha}$, the convolution kernel of $\mR_0(z^{2\alpha})$ is the Fourier transform of the smooth function $\frac{1}{|\xi|^{2\alpha} - z^{2\alpha}}$.  This is a radial function which can be rescaled as
$|z|^{-2\alpha}\big(\big|\frac{\xi}{|z|}\big|^{2\alpha} - (\frac{z}{|z|})^{2\alpha}\big)^{-1}$.  It follows that
\[
\mR_0(z^{2\alpha})(x,y) = |z|^{n-2\alpha} F_{\argz}(|z| r)
\]
where $r = |x-y|$.   Then by the definition of $F(\lambda r)$ in Proposition~\ref{prop:rep},
\be \label{eq:F from Fargz}
F(\lambda r) = (\lambda r)^{n-2\alpha}e^{-i\lambda r} \lim\limits_{\argz \to 0^+} F_{\argz}(\lambda r).
\ee

To determine properties of $F_{\argz}$, it suffices to assume that $|z| = 1$.  Let $|z| = 1$ with $0 < \argz < \frac{\pi}{\alpha}$.  We divide the function $h(\xi) = \frac{1}{|\xi|^{2\alpha} - z^{2\alpha}}$ into three pieces:
\[
\frac{\chi(4 |\xi|)}{|\xi|^{2\alpha} - z^{2\alpha}} + \frac{1 -\chi(\frac{|\xi|}{2})}{|\xi|^{2\alpha} - z^{2\alpha}} + \frac{\chi(\frac{|\xi|}{2}) - \chi(4 |\xi|) }{|\xi|^{2\alpha} - z^{2\alpha}} := h_{ctr}(\xi) + h_{tail}(\xi) + h_{ann}(\xi).
\]

We consider the Fourier transform for large $\rho$ first, where $\rho$ here denotes the Fourier variable. The $h_{ctr}$ piece is supported in the disk $|\xi| < \frac12$ and is smooth except for a polynomial singularity at the origin. 
We write
$$
	h_{ctr}(\xi)=-z^{-2\alpha}\chi(4|\xi|)+\bigg[\frac{1}{|\xi|^{2\alpha} - z^{2\alpha}}+z^{-2\alpha}\bigg]\chi(4|\xi|).
$$
Note that the first term is Schwartz, and the second term satisfies the hypotheses of Lemma~\ref{lem:small xi} with $\gamma=2\alpha$.
It follows that, uniformly in $\arg z$, 
\be\label{eqn:hctr}
	|\nabla^k\widehat{h_{ctr}}(\rho) |\les \la \rho \ra^{-n-2\alpha-k}, \qquad k=0,1,2,\dots
\ee
The existence of the limit as $\arg z\to 0^+$ is clear.  Using this bound in \eqref{eq:F from Fargz} with $\rho=\lambda r$, we conclude that the contribution of $h_{ctr}$ to $F(\lambda r)$ satisfies
$$
	|\partial_\lambda^{N}F_{ctr}(\lambda r)| \les r^{N} (\lambda r)^{n-2\alpha}\la \lambda r\ra^{-n-2\alpha}\les \lambda^{-N} \la \lambda r\ra ^{\frac{n+1}{2}-2\alpha}.
$$
provided that $N\leq \frac{n+1+4\alpha}{2}$.  Further, for  $\lambda r\les 1$ we have 
$$
	|\partial_\lambda^{N}F_{ctr}(\lambda r)| \les  \lambda^{-N}  (\lambda r )^{n-2\alpha+N}.
$$

The $h_{tail}(\xi)$ piece is supported in the region $|\xi| > 2$, with bounds on its derivatives
$|\nabla^k h_{ctr}(\xi)| \les |\xi|^{-2\alpha -k}$.  It follows that for any choice of $M$ and $N$, we have
\[
|\nabla^N\widehat{h_{tail}}(\rho)| \les \rho^{-M} \quad \text{for } \rho >1,
\]
uniformly in $\argz$.

When $\rho<1$, noting that
$n - 2\alpha > 0$, the small $\rho$ behavior of $F_{\argz}(\rho)$ is dominated by the contribution of $h_{tail} $.  In more detail,
\be\label{eqn:ht def}
	h_{tail}(\xi) = \frac{1}{|\xi|^{2\alpha}} -\frac{\chi(\frac{|\xi|}{2})}{|\xi|^{2\alpha}} + z^{2\alpha}\frac{1-\chi(\frac{|\xi|}{2})}{|\xi|^{2\alpha}(|\xi|^{2\alpha}-z^{2\alpha})}.
\ee
The first term contributes a constant multiple of $\rho^{2\alpha-n}$ to $\widehat{h_{tail}}$.
The second term's contribution to $\widehat{h_{tail}}$ is bounded with bounded derivatives by Lemma~\ref{lem:small xi}.
The last term behaves like $|\xi|^{-4\alpha}$ for large $\xi$.  By Lemma~\ref{lem:FT} the  $N^{th}$ derivative of its Fourier transform for $\rho<1$ may be bounded by
\be\label{eqn:htail}
	 \left\{
	\begin{array}{ll}
		1 & 4\alpha-N>n\\
		|\log \rho| & 4\alpha-N=n\\
		\rho^{4\alpha-N-n} & 4\alpha-N <n
	\end{array}
	\right.
\ee
Hence we conclude that for $\rho<1$,
\be\label{eq:htail d}
	|\nabla^N \widehat{h_{tail}}(\rho)|\les 
		\rho^{2\alpha-n-N}.
\ee
Using these bounds, we conclude that the contribution of $h_{tail}$ to $F(\lambda r)$ in \eqref{eq:F from Fargz} satisfies the required bounds
$$
	|\partial_\lambda^{N}F_{tail}(\lambda r)| \les \lambda^{-N} \la \lambda r\ra ^{\frac{n+1}{2}-2\alpha}.
$$
For small $\lambda r$, we may improve the bounds by writing (for $\rho<1$)
$$
	\widehat{h_{tail}}(\rho)=\frac{c_{n,\alpha}}{\rho^{n-2\alpha}}+h_{t2}(\rho),
$$
where
$$
	|\nabla^N h_{t2}(\rho)|\les 
		1 +		\rho^{4\alpha-N-n-}.
$$
So that, from \eqref{eq:F from Fargz}, for $N\geq1$ we have (for $\lambda r<1$)
$$
	|\partial_\lambda^{N}F_{tail}(\lambda r)|\les r^N+\lambda^{-N}(\lambda r)^{n-2\alpha}+\lambda^{-N}(\lambda r)^{2\alpha -}\les  \lambda^{-N}(\lambda r)^{\min(1,n-2\alpha,2\alpha -)}.
$$

We now turn to the contribution of $h_{ann}(\xi)$, which becomes singular as $\argz \to 0^+$, so more care is needed here.  We compare its behavior to a multiple of the resolvent of the Laplacian in order to take advantage of that operator's well known properties.

We write
\[
\frac{1}{|\xi|^{2\alpha} - z^{2\alpha}} = \frac{1}{\alpha z^{2\alpha - 2}(|\xi|^2 - z^2)} + J(z, \xi),
\]
where
\[
J(z,\xi) = \frac{\alpha z^{2\alpha -2}(|\xi|^2 - z^2) - (|\xi|^{2\alpha} - z^{2\alpha})}{\alpha z^{2\alpha - 2}(|\xi|^{2\alpha} - z^{2\alpha})(|\xi|^2 - z^2)}
= \frac{\alpha \big((\frac{|\xi|}{z})^2 - 1\big) - \big((\frac{|\xi|}{z})^{2\alpha}-1\big)}{\alpha z^{2\alpha} \big((\frac{|\xi|}{z})^{2\alpha} - 1\big) \big((\frac{|\xi|}{z})^2 - 1\big)} .
\]

Let $\zeta = \frac{|\xi|}{z}$, which allows some simplification to
\[
J(z, \xi) = \frac{\alpha(\zeta^2 - 1) - (\zeta^{2\alpha} - 1)}{\alpha z^{2\alpha}(\zeta^{2\alpha} - 1)(\zeta^2 - 1)}.
\]
With some abuse of notation, we denote this as $J(z,\zeta)$.
The support of $h_{ann}(\xi)$ consists of an annulus contained in the region $\frac14 \leq |\xi| \leq 4$, and we are still assuming $|z| = 1$.  Thus $\frac14 \leq |\zeta| \leq 4$, and the argument of $\zeta$ is exactly $-\argz$.  Restricting $|\argz| < \frac{\pi}{2\alpha}$ ensures that $J(z, \zeta)$ is a meromorphic function of $\zeta$ in the region corresponding to the support of $h_{ann}(\xi)$, with a possible pole at $\zeta = 1$.

The denominator of $J(z,\zeta)$ vanishes precisely at order $(\zeta - 1)^2$.  Expanding the numerator in a Taylor series around $\zeta =1$ yields the result
\[
\alpha(\zeta^2 - 1) - (\zeta^{2\alpha} - 1) = 2\alpha(1 - \alpha)(\zeta - 1)^2 + O(\zeta - 1)^3,
\]
hence $J(z, \zeta)$ is actually holomorphic in that region. It follows that $J(z, \xi)$ is real-analytic in the support of $\chi(\frac{|\xi|}{2}) - \chi(4|\xi|)$ and varies in an analytic way with $z$ within the range $|\argz| < \frac{\pi}{2\alpha}$.  Now we have the bound
\[
\bigg|\mathcal{F}\bigg[\bigg(\chi\big(\frac{|\cdot|}{2}\big) - \chi(4|\cdot|)\bigg)J(z, \cdot)\bigg](r)\bigg| \les \la \rho\ra ^{-M}
\]
for any $M < \infty$, and furthermore the constants are uniform over $|\argz| < \frac{\pi}{2\alpha}$.  The argument similarly applies to derivatives.  Its contribution to $F(\lambda r)$ in \eqref{eq:F from Fargz} follows as in the argument for $h_{ctr}$.  

The remaining contribution of $h_{ann}(\xi)$ to $F(\lambda r)$ comes from the Fourier transform of
\[
\frac{\chi(\frac{|\xi|}{2}) - \chi(4|\xi|)}{\alpha z^{2\alpha - 2}(|\xi|^2 - z^2)} =
\frac{1}{\alpha z^{2\alpha - 2}(|\xi|^2 - z^2)} 
- \frac{\chi(4|\xi|)}{\alpha z^{2\alpha - 2}(|\xi|^2 - z^2)}
- \frac{1 - \chi(\frac{|\xi|}{2})}{\alpha z^{2\alpha - 2}(|\xi|^2 - z^2)}.
\]
We consider cases when $\rho>1$ and $\rho<1$ separately, considering $\rho>1$ first.
The first term on the right side is a constant multiple of the free resolvent of the Laplacian.  The limit as $\argz \to 0^+$ is known to exist and its kernel is of the form (for $\rho>1$)
$$
	\frac{e^{i\rho}}{\rho^{n-2}} F_1(\rho), \qquad  |\partial_\rho^N F_1(\rho)|\les 
	\la \rho\ra^{\frac{n-3}{2}-N}
$$
for all $N \geq 0$.  Its contribution to $F(\lambda r)$ in \eqref{eq:F from Fargz} satisfies \eqref{Fbounds} when $\lambda r\gtrsim 1$.  The second term is smooth and compactly supported, so its Fourier transform has Schwarz decay for large $\rho$.  The third term is similar to $h_{tail}$, and also gives rise to a kernel with Schwarz decay for large $\rho$ (but a singularity as $\rho \to 0$).  This implies the claim for the contribution of $F_{ann}$ when $\lambda r\gtrsim 1$.

Now, for $\lambda r\les 1$ we consider the left hand side directly.  Its Fourier transform is the resolvent kernel of the Laplacian $(-\Delta - z^2)^{-1}$ mollified by a Schwarz function constructed from $\widehat{\chi}(\rho)$.  As $\argz \to 0^+$, the resolvent kernels converge to the limit $(-\Delta - (1 + i0))^{-1}(r)$.  Then the mollification ensures that the resulting function of $\rho$ is smooth at the origin.  So once again, for $\rho \les 1$ the contribution to $F_{\argz}(\rho)$ is bounded and has bounded derivatives, uniformly up to the limit as $\argz$ approaches zero.  This also satisfies the improved bounds for $\lambda r\les 1$ as in the argument for $h_{ctr}$.

Finally, we observe that $\mR_0^+(\lambda^{2\alpha})- \mR_0^-(\lambda^{2\alpha})$ is $c\lambda^{2-2\alpha}$ times the analogous difference of resolvents of the Laplacian. Both are scaled restrictions in frequency space to the sphere $\{|\xi| = \lambda\}$. The functions $F_+$ and $F_-$ are exactly the same as their counterparts for $(-\Delta - (\lambda^2 \pm i0))^{-1}$, which are known to satisfy the bound  \eqref{Fbounds}, \cite{EGWaveOp2}.

For small $\lambda r$, to establish the second inequality in \eqref{Fbounds2}, we recall the proof of Lemma~2.4 in \cite{EGWaveOp2}.  We may write
$\chi(r)F_{\pm}(r)=\widetilde F(r)/(2\cos r)$ where $\widetilde F(r)$ is entire with bounded derivatives and $\cos(r)\geq \f12$.  From here it is clear that
$$
	|\partial_\lambda^N F_\pm(\lambda r)|\les r^N \la \lambda r\ra^{\frac{1-n}{2}-N},
$$
which implies the claim.

\end{proof}

\begin{proof}[Proof of Proposition~\ref{prop:R0 exp}]
	
	The claim for $\lambda r\gtrsim 1$ follows from the representation in Proposition~\ref{prop:rep}.
	
	When $\lambda r\les 1$, from the proof of Proposition~\ref{prop:rep} we see that the contribution of $h_{ann}$ and $h_{ctr}$ may be bounded by $\lambda^{n-2\alpha}$.  From \eqref{eqn:ht def}, we see that the first term in $h_{tail}$ contributes the Green's function $r^{2\alpha-n}$, the second term contributions $\lambda^{n-2\alpha}$ while the last term's contribution depends on the relative size of $4\alpha$ and $n$.  By \eqref{eqn:htail} (with $N=0$), we see that the contribution is $\lambda^{n-2\alpha}$ when $n<4\alpha$, $\lambda^{n-2\alpha}|\log(\lambda r)|$ when $n=4\alpha$ and $\lambda^{2\alpha}r^{4\alpha-n}$ when $n>4\alpha$.  When $n=4\alpha$ since $\lambda \les 1$ we bound $|\log(\lambda r)| $ by $\log^-$.
	
\end{proof}

\begin{proof}[Proof of Proposition~\ref{prop:LAP}]
	
	We first consider when $V=0$.  The proof of Proposition~\ref{prop:R0 exp} implies that $\lim_{\epsilon\to 0^+} \mR_0(\lambda^{2\alpha}+i\epsilon)$ exists and is equal to $\frac{1}{\alpha}\lambda^{2-2\alpha} R_0^+(\lambda^2)$ plus an error term which consists of the contributions of the Fourier transforms of $h_{ctr}$, $h_{tail}$, $J(z,\xi)$, and
	\be\label{eq:free error}
		- \frac{\chi(4|\xi|)}{\alpha z^{2\alpha - 2}(|\xi|^2 - z^2)}
		- \frac{1 - \chi(\frac{|\xi|}{2})}{\alpha z^{2\alpha - 2}(|\xi|^2 - z^2)}.
	\ee
	By the limiting absorption principle for the classical Schr\"odinger operator, the main piece, $\frac{1}{\alpha}\lambda^{2-2\alpha} R_0^+(\lambda^2)$, satisfies the claim.  To finish the proof it suffices to show that the Fourier transform of the remaining terms are bounded by $\rho^{1-n}$.  Indeed, the contribution to $\mR_0^+(\lambda^{2\alpha})$ is bounded by $\lambda^{n-2\alpha}(\lambda r)^{1-n}=\lambda^{1-2\alpha}r^{1-n}$, which satisfies the claim by boundedness of the fractional integral operators, see Lemma~2.3 in \cite{Jen}.
	
	The contribution $J(z,\xi)$ and the first term in \eqref{eq:free error}  satisfy the claim because they are Schwartz.  From  \eqref{eqn:hctr}, $h_{ctr}(\rho)\les \la \rho\ra^{-n-2\alpha}\les \rho^{1-n}$, so its contribution  to the free resolvent is bounded by $\lambda^{1-2\alpha}r^{1-n}$.  Since $h_{tail}$ has Schwartz decay for large $\rho$ and is at worst $\rho^{2\alpha-n}\les \rho^{1-n}$ for small $\rho$, since $\alpha>\f12$, its contribution is also bounded by $\lambda^{1-2\alpha}r^{1-n}$.  As in the analysis of $h_{tail}$ the second term in \eqref{eq:free error} contributes a Schwartz decay for large $\rho$. For small $\rho$ it contributes either $\rho^{2-n}\les \rho^{1-n}$ for $n>2$ and $|\log \rho|\les \rho^{1-n}$ when $n=2$.  This establishes the claim for $\mR_0^\pm(\lambda^{2\alpha})$.
	
	We now turn to $\mR_V^\pm(\lambda^{2\alpha})$. As in the classical case this follows from the claim for the free resolvent by utilizing the symmetric resolvent identity
	\be\label{eqn:symm resolv id}
		\mR_V^\pm(\lambda^{2\alpha})=\mR_0^\pm(\lambda^{2\alpha})-\mR_0^\pm(\lambda^{2\alpha})v [U+v\mR_0^\pm(\lambda^{2\alpha})v]^{-1}v\mR_0^\pm(\lambda^{2\alpha}),
	\ee
	where
	$v=|V|^{\f12}$, $U=\sgn(V)$ and 
	and by establishing uniform bounds on $[U+v\mR_0^\pm(\lambda^{2\alpha})v]^{-1}$ on $L^{2}$.
	A uniform bound on compact intervals was established in \cite{ZHZ} by applying Agmon's method.  And for large $\lambda$ it's simpler by noting that	
	$\|v\mR_0^\pm(\lambda^{2\alpha})v\|_{L^{2}\to L^{2}}\leq C_V\lambda^{1-2\alpha}<\f12$ provided $\lambda$ is large enough since $\alpha>\f12$.  
	
	For the derivatives, the claim follows from the resolvent identity and the corresponding claims for $\mR_0^\pm(\lambda^{2\alpha})$.  The contribution of the free Schr\"odinger resolvent is well-known.  For the error term we consider the contribution of the second term in \eqref{eq:free error}.  The Fourier transform for large $\rho$ has Schwartz decay and hence satisfies the claim.  For small $\rho$, by Lemma~\ref{lem:FT} it's $j^{th}$ derivative is bounded by $\rho^{2-n-j}$, whose contribution to the $j^{th}$ derivative of the free resolvent by the chain rule is bounded by 
	$$
		\lambda^{n-2\alpha}r^j (\lambda r)^{2-n-j}\chi(\lambda r)\les \lambda^{1-2\alpha -j}r^{1-n}.
	$$
	Since $j\geq 1$ this maps $L^{2,\f12+}\to L^{2,-\f12-}$ with smaller operator norm.  The contribution of $h_{tail}$ can be handled similarly using \eqref{eq:htail d} for small $\rho$ and the Schwartz decay for large $\rho$.
	The contribution of the other terms are simpler.

	The claim for the derivatives of the perturbed resolvent follows from \eqref{eqn:symm resolv id} noting that
	$$
		\partial_\lambda [U+v\mR_0^\pm(\lambda^{2\alpha})v]^{-1}=[U+v\mR_0^\pm(\lambda^{2\alpha})v]^{-1} v\partial_\lambda \mR_0^\pm(\lambda^{2\alpha})v
		[U+v\mR_0^\pm(\lambda^{2\alpha})v]^{-1},
	$$
	and its iterates for higher derivatives. This requires $|v(x)|\les \la x\ra^{-\f12-j-}$.
	
\end{proof}

\section{Proof of Theorem~\ref{thm:main2}}\label{sec:2d}

Employing the Stone's formula, Theorem~\ref{thm:main2} follows by proving
$$
	\sup_{L\geq 1}\sup_{x,y\in\R^2}\bigg|\int_0^\infty e^{it\lambda^{2\alpha}} \lambda^{4\alpha-3} \chi(\lambda/L) [\mR_V^+-\mR_V^-](\lambda^{2\alpha})(x,y)\, d\lambda \bigg| \les \frac{1}{|t|}.
$$
This will be done in two subsections by addressing high energies, when $\lambda\gtrsim 1$, and low energies, when $0<\lambda \ll 1$, separately.

\subsection{High energy}
In this subsection we prove the following proposition.

\begin{prop}\label{prop:2d hi}
	
	Fix $\frac34\leq \alpha<1$ and assume that $|V(x)|\les \la x\ra^{-\beta}$ for some $\beta>\f52$ and that $H$ has no embedded eigenvalues.  Then,
	$$
		\sup_{L\geq 1}\sup_{x,y\in\R^2}\bigg|\int_0^\infty e^{it\lambda^{2\alpha}} \lambda^{4\alpha-3}\widetilde \chi(\lambda) \chi(\lambda/L) \mR_V^\pm(\lambda^{2\alpha})(x,y)\, d\lambda \bigg| \les \frac{1}{|t|}.
	$$
	
\end{prop}
In the high energy argument we don't utilize any cancellation from the difference of the `+' and `-' limiting resolvents.  We drop the superscript and note that the arguments presented handle both cases.
As usual, we iterate the resolvent identity to form a Born series
$$
	\mR_V=\sum_{k=0}^{2K} (-\mR_0V)^k\mR_0+ (\mR_0V)^K\mR_V (V\mR_0)^K,
$$
and bound the finite Born series terms and the tail separately.
We first consider the contribution of the $k^{th}$ Born series term:
\begin{lemma}\label{lem:Born hi}
	
	Fix $k>0$ and $\frac34 \leq \alpha<1$.  If $|V(x)|\les \la x\ra^{-\beta}$ for some $\beta>2\alpha$,  then
	$$
		\sup_{L\geq 1}\sup_{x,y\in\R^2}\bigg|\int_0^\infty e^{it\lambda^{2\alpha}} \lambda^{4\alpha-3}\widetilde \chi(\lambda) \chi(\lambda/L) [(\mR_0V)^k\mR_0](\lambda^{2\alpha})(x,y)\, d\lambda \bigg| \les \frac{1}{|t|}.
	$$
	
\end{lemma}

\begin{proof}
	Define $z_0:=x$ and $z_{k+1}:=y$ and let $r_j=|z_j-z_{j-1}|$ for $j\geq 1$.  Using Proposition~\ref{prop:rep} we need to control
	\be\label{eqn:hi BS}
		\int_{\R^{2k}}\int_0^\infty e^{it\lambda^{2\alpha}+i\lambda R} \lambda^{4\alpha-3}\widetilde \chi(\lambda) \chi(\lambda/L)
		\prod_{j=1}^{k+1} \frac{F(\lambda r_j) }{r_j^{2-2\alpha}} \prod_{i=1}^k V(z_i)  d\lambda\, d\vec z,
	\ee
	where $R=\sum_{j=1}^{k+1}r_j$ and $d\vec z=dz_1dz_2\cdots dz_{k}$.
	
	If $t>0$ there is no critical point of the phase in the support of the integral.  If $t<0$, the critical point of the phase is at 
	$$
		\lambda_0=\bigg(\frac{-R}{2\alpha t}\bigg)^{\frac{1}{2\alpha -1}}.
	$$
	We first consider the case when $\lambda \not \sim \lambda_0$ when $t<0$, or when $t>0$ and $\lambda \gtrsim 1$.  In these cases we may integrate by parts using
	$$
		e^{it\lambda^{2\alpha}+i\lambda R}= \frac{1}{2it\alpha \lambda^{2\alpha-1}+iR}
		\frac{d}{d\lambda} e^{it\lambda^{2\alpha}+i\lambda R}.
	$$
	Note that in the cases being considered, the denominator in magnitude is $\gtrsim |t|\lambda^{2\alpha-1}$.  This is easy to see when $t>0$, when $t<0$ we note that the two terms in the denominator are not comparable in size, so we use the first one.   Now,
	\begin{multline*}
		|\eqref{eqn:hi BS}|=\bigg| \int_{\R^{2k}}\int_0^\infty e^{it\lambda^{2\alpha}+i\lambda R} \frac{d}{d\lambda} \bigg[\frac{ \lambda^{4\alpha-3}\widetilde \chi(\lambda) \chi_{\lambda\not \sim \lambda_0}(\lambda) \chi(\lambda/L)}{2it\alpha \lambda^{2\alpha-1}+iR}
		\prod_{j=1}^{k+1} \frac{F(\lambda r_j) }{r_j^{2-2\alpha}} \prod_{i=1}^k V(z_i)  \bigg]d\lambda\, d\vec z
		\bigg|\\
		\les \frac{1}{|t|} \int_{\R^{2k}}\int_1^\infty \lambda^{2\alpha-3}
		\prod_{j=1}^{k+1} \frac{\la \lambda r_j\ra ^{\frac32-2\alpha}}{r_j^{2-2\alpha}}
		\prod_{i=1}^k |V(z_i)  |
		d\lambda\, d\vec z \les \frac{1}{|t|}\int_{\R^{2k}} \frac{1}{r_j^{2-2\alpha}}
		\prod_{i=1}^k |V(z_i)  | \, d\vec z\les \frac{1}{|t|}.
	\end{multline*}
	Since $\frac34\leq \alpha<1$, we ignore the $\la \lambda r_j\ra$ contribution and the $\lambda$ integral is finite.  Under the decay conditions on $V$ and the conditions on $\alpha$, the spatial integrals are finite.

	Next we consider when $t<0$ and $1\les \lambda \sim \lambda_0$.    Let $r_{j_0}=\max_{1\leq j\leq k+1}r_j$ and apply Van der Corput.   Note that $|\frac{d^2}{d\lambda^2}  (t\lambda^{2\alpha}+\lambda R)|\approx |t\lambda_0^{2\alpha-2}|$, so that
	\begin{multline*}
			|\eqref{eqn:hi BS}|\les  \int_{\R^{2k}} |t\lambda_0^{2\alpha-2}|^{-\f12} \int_{\lambda \sim \lambda_0}  \bigg|
			 \frac{d}{d\lambda} \bigg[ \lambda^{4\alpha-3}\widetilde \chi(\lambda)  \chi(\lambda/L)
			\prod_{j=1}^{k+1} \frac{F(\lambda r_j) }{r_j^{2-2\alpha}} \prod_{i=1}^k V(z_i)  \bigg] \bigg|d\lambda\,
			d\vec z\\
			\les \int_{\R^{2k}} |t\lambda_0^{2\alpha-2}|^{-\f12} \int_{\lambda \sim \lambda_0}  \bigg|\lambda^{4\alpha-4}\widetilde \chi(\lambda)  \chi(\lambda/L)
			\prod_{j=1}^{k+1} \frac{\la \lambda r_j\ra^{\f32-2\alpha} }{r_j^{2-2\alpha}} \prod_{i=1}^k V(z_i)  \bigg|d\lambda\,
			d\vec z\\
			\les \int_{\R^{2k}} |t\lambda_0^{2\alpha-2}|^{-\f12} \int_{\lambda \sim \lambda_0}  \bigg|\frac{\lambda^{2\alpha-\f52}}{r_{j_0}^\f12}\widetilde \chi(\lambda)  \chi(\lambda/L)
			\prod_{j=1, j\neq j_0}^{k+1} \frac{1 }{r_j^{2-2\alpha}} \prod_{i=1}^k V(z_i)  \bigg|d\lambda\,
			d\vec z\\
			\les \int_{\R^{2k}} |t|^{-\f12}  \bigg|\frac{\lambda_0^{\alpha-\f12}}{r_{j_0}^\f12}
			\prod_{j=1, j\neq j_0}^{k+1} \frac{1 }{r_j^{2-2\alpha}} \prod_{i=1}^k V(z_i)  \bigg|
			d\vec z\les \frac{1}{|t|} .
	\end{multline*}
	Here we used that $\la \lambda r_j\ra ^{\f32-2\alpha}\les 1$ for any $j\neq j_0$ and $\la \lambda r_{j_0}\ra ^{\f32-2\alpha}\les (\lambda r_{j_0}) ^{\f32-2\alpha}$, and the spatial integrals are bounded as above.
	
\end{proof}

We now consider the final term in the Born series.

	\begin{lemma}\label{lem:tail hi}
		
		Fix $K>0$ large enough, $\frac34 \leq \alpha<1$, and assume that $|V(x)|\les \la x\ra^{-\beta}$ for some $\beta>\f52$.  Then, if $H$ has no embedded eigenvalues,
		$$
		\sup_{L\geq 1}\sup_{x,y\in\R^2}\bigg|\int_0^\infty e^{it\lambda^{2\alpha}} \lambda^{4\alpha-3}\widetilde \chi(\lambda) \chi(\lambda/L) [(\mR_0V)^K\mR_V(V\mR_0)^K](\lambda^{2\alpha})(x,y)\, d\lambda \bigg| \les \frac{1}{|t|}.
		$$
		
	\end{lemma}

	\begin{proof}
		
		Let
		$$
			a_{x,y}(\lambda)=\lambda^{4\alpha-3}\widetilde \chi(\lambda) \chi(\lambda/L) e^{-i\lambda |x|}e^{-i\lambda |y|} [(\mR_0V)^K\mR_V(V\mR_0)^K](\lambda^{2\alpha})(x,y)
		$$
		We claim that
		$$
			|\partial_\lambda^j a_{x,y}(\lambda)|\les \lambda^{-2} \la x\ra^{-\f12} \la y\ra^{-\f12}, \qquad j=0,1.
		$$
		For $j=0$, using the limiting absorption principle in Proposition~\ref{prop:LAP} for both $\mR_V$ and $\mR_0$, we have
		$$
			|a_{x,y}(\lambda)|\les \lambda^{4\alpha-3} (\lambda^{1-2\alpha})^{2K-1} \|\mR_0(x,\cdot) V(\cdot)\la \cdot\ra^{\f12+}\|_{L^2}\| \la \cdot\ra^{\f12+} V(\cdot)\mR_0(\cdot,y)\|_{L^2}
			\les \lambda^{-2}\la x\ra^{-\f12}\la y\ra^{-\f12}.
		$$
		Where in the last inequality we used \eqref{Fbounds} to see that
		$$
			|\mR_0(\lambda^{2\alpha})(x,x_1)V(x_1)\la x_1\ra^{\f12+}|\les \frac{\la \lambda |x-x_1|\ra^{\f32-2\alpha}}{|x-x_1|^{2-2\alpha}\la x_1\ra^{\beta-\f12-}}\les \frac{\lambda^{\f32-2\alpha}}{|x-x_1|^{\f12}\la x_1\ra^{\beta-\f12-}}.
		$$
		So that if $\beta>\f32$, the $L^2$ norm is bounded by $\la x\ra^{-\f12}$.  For $j=1$ we note that for the leading and lagging resolvents we have
		$$
			\partial_\lambda \big[e^{-i\lambda |x|}\mR_0(\lambda^{2\alpha})(x,x_1) \big]=\partial_\lambda \big[e^{-i\lambda (|x|-|x-x_1|)}\frac{F(\lambda |x-x_1|)}{|x-x_1|^{2-2\alpha}} \big].
		$$
		If the derivative hits the phase we bound $|x|-|x-x_1|$ by $\la x_1\ra$.
		the remaining part of the proof is similar but requires $\beta>\f52$ since one needs to account for a larger weight.

		It suffices to consider
		$$
			\int_1^\infty e^{it\lambda^{2\alpha}+i\lambda(|x|+|y|)} a_{x,y}(\lambda)\, d\lambda.
		$$
		It is clear that this is bounded by one uniformly in $x,y$.  For the time decay, as in the proof of Lemma~\ref{lem:Born hi}, we consider the cases of $\lambda \sim \lambda_0=(-\frac{|x|+|y|}{2\alpha t})^{\frac1{2\alpha-1}}$ and $\lambda \not \sim \lambda_0$.
		
		In the case when $\lambda \not \sim \lambda_0$, one integration by parts results in
		$$
			\bigg|\int_1^\infty e^{it\lambda^{2\alpha}+i\lambda(|x|+|y|)} a_{x,y}(\lambda)\, d\lambda\bigg|\les \frac{1}{|t|} \int_1^\infty \lambda^{-2-2\alpha}\, d\lambda \les \frac{1}{|t|}.
		$$
		When $\lambda \sim \lambda_0$, by Van der Corput, we have
		$$
			\bigg|\int_1^\infty e^{it\lambda^{2\alpha}+i\lambda(|x|+|y|)} a_{x,y}(\lambda)\, d\lambda\bigg|\les |t|^{-\f12} \la x\ra^{-\f12}\la y\ra^{-\f12} |\lambda_0|^{1-\alpha} \int_{\lambda \sim \lambda_0} \lambda^{-2}\, d\lambda
			\les \frac{\la x\ra^{-\f12}\la y\ra^{-\f12}}{|t|^{\f12}}\les |t|^{-1}.
		$$
		Where we used that $\lambda_0\gtrsim 1$, which in particular implies that $|t|\les |x|+|y|$.

	\end{proof}

Proposition~\ref{prop:2d hi} now follows from the Born series expansion and Lemmas~\ref{lem:Born hi} and \ref{lem:tail hi}.

\subsection{Low energy}

We now consider the low energy, when $0<\lambda<\lambda_0$ for a sufficiently small constant $\lambda_0\ll1$.
We utilize the symmetric resolvent identity,
\be\label{eqn:sym res}
	\mR_V^\pm(\lambda^{2\alpha})=\mR_0^\pm(\lambda^{2\alpha})-\mR_0^\pm(\lambda^{2\alpha})v M_\pm^{-1}(\lambda) v\mR_0^\pm(\lambda^{2\alpha}),
\ee
where $v=|V|^{\f12}$, $U=\sgn(V)$ and $M^\pm(\lambda)=U+v\mR_0^\pm(\lambda^{2\alpha})v$.  By the assumption that zero is regular, we have that $M^\pm(0)=U+v\mR_0^\pm(0)v$ is invertible on $L^2(\R^2)$.  The following proposition finishes the proof of Theorem~\ref{thm:main2}.

\begin{prop}\label{prop:low tail1}
	
	If zero is a regular point of $H$ and $|V(x)|\les \la x\ra^{-\beta}$ for some $\beta>4$, then
	$$
	\sup_{x,y\in\R^2}\bigg|\int_0^\infty e^{it\lambda^{2\alpha}} \lambda^{4\alpha-3} \chi(\lambda)  [\mR_V^+-\mR_V^-](\lambda^{2\alpha})(x,y)\, d\lambda \bigg| \les \frac{1}{|t|}.
	$$
	
\end{prop}

We first prove the following lemma.

\begin{lemma}\label{lem:Minv 2}
	
	For sufficiently small $\lambda_0$, if $|V(x)|\les \la x\ra^{-\beta}$ for some $\beta>3$ and zero is regular, then the operators $M^\pm(\lambda)$ are invertible on $L^2$.  Furthermore,
	$$
		\bigg\|\sup_{0<\lambda<\lambda_0} |M_\pm^{-1}(\lambda)|+\lambda^{1-} |\partial_\lambda M_\pm^{-1}(\lambda)|\,\bigg\|_{L^2\to L^2}<\infty.
	$$
	In addition,
	$$
		\bigg\|\sup_{0<\lambda<\lambda_0} \lambda^{2\alpha-2}|[M_+^{-1}-M_-^{-1}](\lambda)|+\lambda^{2\alpha-1-} |\partial_\lambda [M_+^{-1}-M_-^{-1}](\lambda)|\,\bigg\|_{L^2\to L^2}<\infty.
	$$
	
\end{lemma}

\begin{proof}
	
	From Proposition~\ref{prop:R0 exp}, we write
	$$
		M^+(\lambda)=U+v\mR_0^+(0)v+\mathcal E(\lambda), \qquad \mathcal E(\lambda)=O(\lambda^{2-2\alpha} |v|(x)|v|(y)).
	$$
	If $|V(x)|\les \la x\ra^{-\beta}$ for $\beta>2$, we have
	$$
		\sup_{0<\lambda<\lambda_0} |\mathcal E(\lambda)|\les \lambda_0^{2-2\alpha}\la x\ra^{-1-}\la y\ra^{-1-},
	$$
	which is bounded on $L^2$.
	
	Let $T_0=U+v\mR_0^+(0)v$, by a standard argument $T_0^{-1}$ is absolutely bounded.  Namely, we since $|[v\mR_0^+v](0)(x,y)|\les   \la x\ra^{-\frac{\beta}2} I_{2\alpha}(x,y) \la y\ra^{-\frac{\beta}2}$ with $I_{2\alpha}$ the fractional integral operator.  As in the proof of Lemma~\ref{lem:spec small a}, the resulting operator is bounded on $L^2$, and in this case is even Hilbert-Schmidt.  By the resolvent identity, we have $T_0^{-1}=U[I-(v\mR_0^+(0)v)T_0^{-1}]$. The first term, $U$ is clearly absolutely bounded while $v\mR_0^+(0)vT_0^{-1}$ is the composition of a Hilbert-Schmidt and bounded operator and is hence Hilbert-Schmidt and consequently absolutely bounded.
	
	So, the claim follows by a Neumann series expansion  for sufficiently small $\lambda_0$.  For the derivative, we use the resolvent identity to write
	$$
		\partial_\lambda M_+^{-1}(\lambda)= M_+^{-1}(\lambda)v\partial_\lambda \mR_0^+(\lambda^{2\alpha})v M_+^{-1}(\lambda).
	$$
	We note that, by Proposition~\ref{prop:R0 exp} we have
	\begin{multline*}
		\lambda^{1-}|\partial_\lambda \mR_0^+(\lambda^{2\alpha})(x,y)|\les \lambda^{1-}|x-y|\frac{\la \lambda |x-y|\ra^{\f32-2\alpha}}{|x-y|^{2-2\alpha}}+\lambda^{1-} \lambda^{-1}(\lambda |x-y|)^{0+}\frac{\la \lambda |x-y|\ra^{\f32-2\alpha}}{|x-y|^{2-2\alpha}}\\
		\les \frac{1}{|x-y|^{2-2\alpha-}}+|x-y|^{\f12}.
	\end{multline*}
	The last bound is seen by considering the cases of $\lambda|x-y|<1$ and $\lambda |x-y|\geq 1$ separately.  Then
	$$
		\sup_{0<\lambda<\lambda_0} \lambda^{1-}|\partial_\lambda[ v\mR_0^+(\lambda^{2\alpha})v](x,y)|
	$$
	is bounded on $L^2$ provided $\beta>3$.
	
	The second claim follows from the resolvent identity:
	\begin{multline*}
		[M_+^{-1}-M_{-}^{-1}](\lambda)=-M_{+}^{-1}(\lambda)v[\mR_0^+(\lambda^{2\alpha})-\mR_0^-(\lambda^{2\alpha})]vM_{-}^{-1}(\lambda)\\
		=-\lambda^{2-2\alpha}M_{+}^{-1}(\lambda)v[e^{i\lambda r}F_+(\lambda r)+e^{-i\lambda r} F_{-}(\lambda r)]vM_{-}^{-1}(\lambda)
	\end{multline*}
	The claim now follows as above from \eqref{Fbounds2} and the bounds on $M_{\pm}^{-1}$.
	
\end{proof}

\begin{proof}[Proof of Proposition~\ref{prop:low tail1}]

	It suffices to consider the contribution of the following operators to $\mR_V^+-\mR_V^-$ in \eqref{eqn:sym res}:
	$\mR_0^-(\lambda^{2\alpha})v M_+^{-1}(\lambda) v[\mR_0^+(\lambda^{2\alpha})-\mR_0^-(\lambda^{2\alpha})]$ and $\mR_0^-(\lambda^{2\alpha})v [M_+^{-1}-M_-^{-1}](\lambda) v\mR_0^+(\lambda^{2\alpha})$.  In both cases we consider an operator $\Gamma(\lambda)$ of the form where
	$$
	\widetilde \Gamma:=\sup_{0<\lambda<\lambda_0}( |\Gamma(\lambda)|+\lambda^{1-} |\Gamma'(\lambda)|)
	$$
	is bounded on $L^2$.  By Lemma~\ref{lem:Minv 2} both $M_+^{-1}(\lambda)$ and $\lambda^{2\alpha-2 }[M_+^{-1}-M_-^{-1}](\lambda)$ satisfy this bound.
	
	By Proposition~\ref{prop:R0 exp} and the definition of $\Gamma(\lambda)$ above, we need to control
	\be\label{eqn:low tail goal}
		\int_0^1  e^{it\lambda^{2\alpha}+i\lambda (|x|\mp |y|)} \lambda^{2\alpha-1}\chi(\lambda)  a_{x,y}(\lambda) \, d\lambda,
	\ee
	where (with $r_1=|x-z_1|$ and $r_2=|z_2-y|$)
	\be
		a_{x,y}(\lambda)=\int_{\R^4}  e^{i\lambda (r_1- |x|\pm (r_2-|y|))} \frac{F(\lambda r_1)}{r_1^{2-2\alpha}} v(z_1)\Gamma(\lambda)(z_1,z_2)v(z_2)
		\bigg(\frac{F(\lambda r_2)}{ r_2^{2-2\alpha}}+F_{\pm}(\lambda r_2) \bigg)\, dz_1 \, dz_2.
	\ee
	Note that, using $|F(\cdot)|,|F_\pm(\cdot)|\les 1$, we have
	$$
		|a_{x,y}(\lambda)|\les \bigg\| \frac{v(\cdot)}{|x-\cdot|^{2-2\alpha}}\bigg\|_{L^2} \| \widetilde \Gamma \|_{L^2\to L^2}\bigg\| v(\cdot)\bigg(1+\frac{1}{|\cdot-y|^{2-2\alpha}}\bigg)\bigg\|_{L^2}\les 1,
	$$
	uniformly in $x,y\in \R^2$, provided that $\beta>2$.
	
	On the other hand, assuming $|y|>|x|$ and 
	using $|F(\lambda r_2)|\les \lambda^{\frac32-2\alpha} r_2^{\frac32-2\alpha}$ and $|F_{\pm}(\lambda r_2)|\les \lambda^{-\f12}r_2^{-\f12}$, we have (provided $\beta>2$)
	$$
		|a_{x,y}(\lambda)|\les \lambda^{-\f12} \bigg\| \frac{v(\cdot)}{|x-\cdot|^{2-2\alpha}}\bigg\|_{L^2} \| \widetilde \Gamma \|_{L^2\to L^2}\bigg\| \frac{v(\cdot)}{|\cdot-y|^{\f12}}\bigg\|_{L^2}\les \lambda^{-\f12} \la y\ra^{-\f12},
	$$
	where we used that $\frac32-2\alpha>-\f12$ and $0<\lambda <1$.  The case of $|x|>|y|$ follows similarly with a bound of $\lambda^{-\f12}\la x\ra^{-\f12}$.	 So that, 
	$|a_{x,y}(\lambda)|\les \min(1,\lambda^{-\f12} (\la x\ra +\la y\ra)^{-\f12})$. 	
	
	Further, using \eqref{Fbounds}, assuming $|y|>|x|$ and $\beta>4$ we have
	$$
			|\partial_\lambda a_{x,y}(\lambda)|\les \lambda^{-\f32} \bigg\| \frac{v(\cdot)\la \cdot\ra}{|x-\cdot|^{2-2\alpha}}\bigg\|_{L^2} \| \widetilde \Gamma \|_{L^2\to L^2}\bigg\| \frac{v(\cdot)\la \cdot\ra}{|\cdot-y|^{\f12}}\bigg\|_{L^2}\les \lambda^{-\f32} (\la x\ra +\la y\ra)^{-\f12}.
	$$
	Here we used that $|r_1-|x||\les \la z_1\ra$, and note that the case of $|x|>|y|$ follows similarly.

	On the other hand, using \eqref{Fbounds2}, we have $|\partial_\lambda F(\lambda r)|,|\partial_\lambda F_\pm(\lambda r)|\les \lambda^{-1}(\lambda r)^{0+}$ we have 
	$$
		|\partial_\lambda a_{x,y}(\lambda)|\les \lambda^{-1+} \bigg\| \frac{v(\cdot)\la \cdot\ra}{|x-\cdot|^{2-2\alpha-}}\bigg\|_{L^2} \| \widetilde \Gamma \|_{L^2\to L^2}\bigg\| v(\cdot)\la \cdot\ra\la \cdot-y\ra^{0+}\bigg( 1+\frac{1}{|\cdot-y|^{2-2\alpha}}\bigg)\bigg\|_{L^2}\les \lambda^{-1}(\lambda \la y\ra)^{0+}.
	$$
	So that
	 $|\partial_\lambda a_{x,y}(\lambda)|\les \lambda^{-1}\min((\lambda(\la x\ra + \la y\ra))^{0+},(\lambda (\la x\ra +\la y\ra))^{-\f12})$.  We note that this is an $L^1$ function of $\lambda$ uniformly in $x,y$.

	When $\lambda \sim \lambda_0$, by Van der Corput we have
	\begin{multline*}
		|t|^{-\frac{1}{2}}\lambda_0^{1-\alpha} \int_{\lambda \sim \lambda_0} |\partial_\lambda [\lambda^{2\alpha-1} a_{x,y}(\lambda)]|\, d\lambda \les |t|^{-\f12}\lambda_0^{1-\alpha} \int_{\lambda \sim \lambda_0} \lambda^{2\alpha-\f52} (\la x\ra +\la y\ra)^{-\f12}\, d\lambda\\
		\les |t|^{-\f12} \lambda_0^{\alpha-\f12}(\la x\ra +\la y\ra)^{-\f12}\les 
		|t|^{-\f12}  \bigg(\frac{|x|\mp |y|}{|t|}\bigg)^{\f12}(\la x\ra +\la y\ra)^{-\f12}\les |t|^{-1}.
	\end{multline*}
	
	When $\lambda \not \sim \lambda_0$ we  integrate by parts to see the bound
	$$
		\bigg|\frac{d}{d\lambda}\bigg[ \frac{\lambda^{2\alpha-1}}{t\lambda^{2\alpha-1}+(|x|\mp |y|)}a_{x,y}(\lambda)\bigg]\bigg|\les |\sup_{\lambda} a_{x,y}(\lambda)|\, \bigg|\frac{d}{d\lambda} \frac{\lambda^{2\alpha-1}}{t\lambda^{2\alpha-1}+(|x|\mp |y|)}\bigg|+
		\frac{1}{|t|}|\partial_\lambda a_{x,y}(\lambda)|.
	$$
	The contribution of the first term is bounded by its supremum on the support of the integral since it changes sign finitely many times, which yield a bound of $|t|^{-1}$.  The contribution of the second term is also $|t|^{-1}$ since $\partial_\lambda a_{x,y}(\lambda)$ is in $L^1_\lambda$ uniformly in $x,y$.

\end{proof}

\section{Proof of Theorem~\ref{thm:main n}}\label{sec:nd}

As in the $n=2$ argument, employing the Stone's formula, Theorem~\ref{thm:main n} follows by proving
$$
	\sup_{L\geq 1}\sup_{x,y\in\R^n}\bigg|\int_0^\infty e^{it\lambda^{2\alpha}} \lambda^{2\alpha-1} \chi(\lambda/L) [\mR_V^+-\mR_V^-](\lambda^{2\alpha})(x,y)\, d\lambda \bigg| \les |t|^{-\frac{n}{2\alpha}}.
$$
This will again be done in two subsections by addressing high energies and low energies separately.

\subsection{High energy}

In this section we prove the following bound.
\begin{prop}\label{prop:high energy n}
	
	Fix  $\frac{n+1}{4} \leq \alpha<\frac{n}2$.  If $H$ has no embedded eigenvalues and $|V(x)|\les \la x\ra^{-\beta}$ for some $\beta>\max(\frac{n+1}{2},\frac{n}{2\alpha}+\f52)$,  then
	$$
	\sup_{L\geq 1}\sup_{x,y\in\R^n}\bigg|\int_0^\infty e^{it\lambda^{2\alpha}} \lambda^{2\alpha-1}\widetilde \chi(\lambda) \chi(\lambda/L) \mR_V^\pm(\lambda^{2\alpha})(x,y)\, d\lambda \bigg| \les |t|^{-\frac{n}{2\alpha}}.
	$$
	
\end{prop}

The following lemma takes care of the contribution of the $k^{th}$ Born series term.
\begin{lemma}\label{lem:Born hi big n}
	
	Fix $k>0$ and $\frac{n+1}{4} \leq \alpha<\frac{n}2$.  If $|V(x)|\les \la x\ra^{-\beta}$ for some $\beta>\frac{n+1}{2}$,  then
	$$
	\sup_{L\geq 1}\sup_{x,y\in\R^n}\bigg|\int_0^\infty e^{it\lambda^{2\alpha}} \lambda^{2\alpha-1}\widetilde \chi(\lambda) \chi(\lambda/L) [(\mR_0V)^k\mR_0](\lambda^{2\alpha})(x,y)\, d\lambda \bigg| \les |t|^{-\frac{n}{2\alpha}}.
	$$
	
\end{lemma}

\begin{proof}
This contribution of the $k^{th}$ term of the Born series to the Stone's formula representation is an integral of the form
\be\label{eqn:BS hi hi d}
	\int_{\R^{2k}}\int_0^\infty e^{it\lambda^{2\alpha}+i\lambda R}\lambda^{2\alpha-1}\widetilde \chi(\lambda) \chi(\lambda/L) \prod_{j=1}^{k+1} \frac{F(\lambda r_j)}{r_j^{n-2\alpha}} \prod_{i=1}^k V(x_k) \, d\lambda d\vec x,
\ee
where $r_j=|x_j-x_{j-1}|$ and $R=\sum r_j$.  We consider two regimes, first if $\lambda \not \sim \lambda_0=(\frac{-R}{2\alpha t})^{\f1{2\alpha-1}}$ we may integrate by parts twice without boundary terms.  Define $\chi_{\lambda_0}(\lambda)$ to be a smooth cut-off to the neighborhood $\frac12\lambda_0<\lambda <2\lambda_0$, and $\widetilde \chi_{\lambda_0}(\lambda)=1-\chi_{\lambda_0}(\lambda)$.  
Denoting $\lambda^{2\alpha}+\frac{R}{t}:=\phi(\lambda)$, we see (for $|t|\gtrsim 1$)
\begin{multline*}
	|\eqref{eqn:BS hi hi d}| \les \frac{1}{|t|^2} \int_{\R^{2k}}\int_1^\infty  
	\bigg| \partial_\lambda \bigg( \frac{1}{\phi'(\lambda)} \partial_\lambda \bigg(\frac{\lambda^{2\alpha-1}}{\phi'(\lambda)}
	\widetilde \chi(\lambda)\widetilde \chi_{\lambda_0}(\lambda) \chi(\lambda/L) \prod_{j=1}^{k+1} \frac{F(\lambda r_j)}{r_j^{n-2\alpha}}\bigg)\bigg) \prod_{i=1}^k V(x_k) \bigg| \, d\lambda d\vec x\\
	\les \frac{1}{|t|^2}\int_{\R^{2k}} \int_1^\infty \lambda^{-2\alpha-3-k(\frac{n+1}{2}-2\alpha)} \, d\lambda \prod_{j=1}^{k+1} \frac{1}{r_j^{\frac{n-1}{2}}} \prod_{i=1}^k |V(x_k)| \, d\vec x.
\end{multline*}
Here we used that $|\phi'(\lambda)|=|2\alpha \lambda^{2\alpha-1}+R/t|=|2\alpha (\lambda^{2\alpha-1}-\lambda_0^{2\alpha-1})|\gtrsim \lambda^{2\alpha-1}$, \eqref{Fbounds}, and that all derivatives are bounded by division by $\lambda$.  Since $\frac{n+1}{2}-2\alpha<0$ and $-2\alpha-3<-1$, the integral converges.	The spatial integrals are bounded provided $|V(x)|\les \la x\ra^{-\frac{n+1}{2}-}$ by standard arguments.  This suffices to ensure we get a large time bound of size $|t|^{-n/2\alpha}$ as desired.

For small times, we need to consider cases based on $\alpha$ and $n$.  First, if $\alpha>\frac{n+1}{4}$ we can integrate by parts once to see its contribution to \eqref{eqn:BS hi hi d} is bounded by
\begin{multline*}
	\frac{1}{|t|} \int_{\R^{2k}}\int_1^\infty  
	\bigg|  \partial_\lambda \bigg(\frac{\lambda^{2\alpha-1}}{\phi'(\lambda)}
	\widetilde \chi(\lambda) \widetilde \chi_{\lambda_0}(\lambda) \chi(\lambda/L) \prod_{j=1}^{k+1} \frac{F(\lambda r_j)}{r_j^{n-2\alpha}}\bigg) \prod_{i=1}^k V(x_k) \bigg| \, d\lambda d\vec x\\
	\les \frac{1}{|t|}\int_{\R^{2k}} \int_1^\infty \lambda^{-1} \, d\lambda \prod_{j=1}^{k+1} \frac{\la \lambda r_j \ra^{\frac{n+1}{2}-2\alpha }}{r_j^{n-2\alpha}} \prod_{i=1}^k |V(x_k)| \, d\vec x\\
	\les \frac{1}{|t|}\int_{\R^{2k}} \int_1^\infty \lambda^{-1-(k+1)(\frac{n+1}{2}-2\alpha)} \, d\lambda \prod_{j=1}^{k+1} \frac{1 }{r_j^{\frac{n-1}{2}}} \prod_{i=1}^k |V(x_k)| \, d\vec x.
\end{multline*}
Since $\frac{n+1}{2}-2\alpha>0$, the $\lambda$ integral converges.  The spatial integrals converge provided $|V(x)|\les \la x\ra^{-\frac{n+1}{2}-}$.  This suffices for small $|t|\les 1$.

Finally, if $|t|<1$ and $\alpha=\frac{n+1}{4}$ we integrate by parts twice and use $|t\phi'(\lambda)|=|2\alpha t\lambda^{2\alpha-1}+R| \gtrsim (|t|\lambda^{2\alpha-1})^{\f12} R^{\f12}$ to see that 
\begin{multline*}
	\frac{1}{|t|^2} \int_{\R^{2k}}\int_1^\infty  
	\bigg| \partial_\lambda \bigg( \frac{1}{\phi'(\lambda)} \partial_\lambda \bigg(\frac{\lambda^{2\alpha-1}}{\phi'(\lambda)}
	\widetilde \chi(\lambda) \chi(\lambda/L) \prod_{j=1}^{k+1} \frac{F(\lambda r_j)}{r_j^{n-2\alpha}}\bigg)\bigg) \prod_{i=1}^k V(x_k) \bigg| \, d\lambda d\vec x\\
	\les \frac{1}{|t|^{\f32}}\int_{\R^{2k}} \int_1^\infty \frac{\lambda^{-\f32}}{R^{\f12}} \, d\lambda \prod_{j=1}^{k+1} \frac{1}{r_j^{\frac{n-1}{2}}} \prod_{i=1}^k |V(x_k)| \, d\vec x\\
	\les \frac{1}{|t|^{\f32}}\int_{\R^{2k}} \prod_{j=1}^{k+1} \frac{1}{r_j^{\frac{n-1}{2}+\frac{1}{2k+2}}} \prod_{i=1}^k |V(x_k)| \, d\vec x.
\end{multline*}
Since $\frac{n}{2\alpha}=\frac{2n}{n+1}=2-\frac2{n+1}\geq \frac32$, this suffices for small $|t|$.  The spatial integrals converge when $|V(x)|\les \la x\ra^{-\frac{n+1}{2}-}$.

We now consider when $\lambda$ is in a neighborhood of the critical point.  In this case, for $\lambda_0$ to be in the support of $\widetilde \chi(\lambda)$ we must have $|t|\les R$.  We proceed by integrating by parts once, this time without combining the phases to express its contribution to \eqref{eqn:BS hi hi d} as
\begin{multline*} 
	\int_{\R^{2k}}\int_0^\infty e^{it\lambda^{2\alpha}}\lambda^{2\alpha-1}\widetilde \chi(\lambda) \chi_{\lambda_0}(\lambda) \chi(\lambda/L) e^{i\lambda R} \prod_{j=1}^{k+1} \frac{F(\lambda r_j)}{r_j^{n-2\alpha}} \prod_{i=1}^k V(x_k) \, d\lambda d\vec x\\
	=\frac{1}{2\alpha it}\int_{\R^{2k}}\int_0^\infty e^{it\lambda^{2\alpha}} \partial_\lambda \bigg(\widetilde \chi(\lambda) \chi_{\lambda_0}(\lambda) \chi(\lambda/L) e^{i\lambda R}\prod_{j=1}^{k+1} \frac{F(\lambda r_j)}{r_j^{n-2\alpha}}\bigg) \prod_{i=1}^k V(x_k) \, d\lambda d\vec x\\
	:=\frac{1}{t} \int_{\R^{2k}} \int_0^\infty e^{it\lambda^{2\alpha}+i\lambda R} a_{\vec x}(\lambda) \, d\lambda d\vec x,
\end{multline*}
with
\begin{multline*}
	a_{\vec x}(\lambda)=\frac{1}{2\alpha i}\bigg[\partial_\lambda \bigg(\widetilde \chi(\lambda) \chi_{\lambda_0}(\lambda) \chi(\lambda/L) \prod_{j=1}^{k+1} \frac{F(\lambda r_j)}{r_j^{n-2\alpha}}\bigg) \\
	+iR \bigg(\widetilde \chi(\lambda) \chi_{\lambda_0}(\lambda) \chi(\lambda/L) \prod_{j=1}^{k+1} \frac{F(\lambda r_j)}{r_j^{n-2\alpha}}\bigg)  \bigg]
	\prod_{i=1}^k V(x_k) .
\end{multline*}

On the support of $\chi_{\lambda_0}(\lambda)$, we may employ Van der Corput to bound by
\begin{multline*}
	\frac{1}{|t|^{\f32}}\int_{\R^{2k}} \lambda_0^{1-\alpha}  \int_{\lambda \sim \lambda_0}  |\partial_\lambda a_{\vec x}(\lambda)|\, d\lambda \, d\vec x\\ 
	\les \frac{1}{|t|^{\f32}} \int_{\R^{2k}} \lambda_0^{1-\alpha}  \int_{\lambda \sim \lambda_0} (R+\lambda^{-1})\lambda^{-1} \prod_{j=1}^{k+1} \frac{\la \lambda r_j\ra^{\frac{n+1}{2}-2\alpha}}{r_j^{n-2\alpha}}\, d\lambda \prod_{i=1}^k |V(x_k)|\, d\vec x\\
	\les \frac{1}{|t|^{\f32}} \int_{\R^{2k}}	(R+1)\lambda_0^{1-\alpha-(k+1)(2\alpha-\frac{n+1}{2})}\prod_{j=1}^{k+1} \frac{1}{ r_j^{\frac{n-1}{2} }} \prod_{i=1}^k |V(x_k)|\, d\vec x.
\end{multline*}
Where we used that $\lambda_0\gtrsim1$ in the last bound.
From here, we must consider cases based on the relative sizes of $n$ and $\alpha$, specifically whether $\frac{n}{2\alpha}\leq \f32$ or $\frac {n}{2\alpha}> \f32$.   We first consider when $\frac{n}{2\alpha}\leq \f32$, so that $0\leq \frac{3}{2}-\frac{n}{2\alpha}<\f12$.  Note that $\frac{n}{3}\leq \alpha$ in this case. Using that $|t|\approx \lambda_0^{2\alpha-1}/R$ we have
$$
	\frac{1}{|t|^{\f32}}\approx \frac{1}{|t|^{\frac{n}{2\alpha}}} \bigg(\frac{\lambda_0^{2\alpha-1}}{R}
	\bigg)^{\f32-\frac{n}{2\alpha}}=\frac{1}{|t|^{\frac{n}{2\alpha}}}\frac{\lambda_0^{(2\alpha-1)(\f32-\frac{n}{2\alpha})}}{R^{{\f32-\frac{n}{2\alpha}}}}
$$

Now, since $k\geq 1$, we have the total power of $\lambda_0$ is
$$
	(2\alpha-1)\bigg(\frac32-\frac{n}{2\alpha}\bigg)+1-\alpha-(k+1)\bigg(2\alpha-\frac{n+1}{2}\bigg)\leq \f12+\frac{n}{2\alpha}-2\alpha\leq 0.
$$
From this, we can see that 
$$
	\frac{1}{|t|^{\f32}} 	(R+1)\lambda_0^{1-\alpha-(k+1)(2\alpha-\frac{n+1}{2})}\prod_{j=1}^{k+1} \frac{1}{ r_j^{\frac{n-1}{2} }}
	\les\frac{1}{|t|^{\frac{n}{2\alpha}}}\frac{1}{R^{{\f32-\frac{n}{2\alpha}}}}	(R+1) \prod_{j=1}^{k+1} \frac{1}{ r_j^{\frac{n-1}{2} }}
$$
Let $r_{j_0}=\max(r_j)$ and note that $R\approx r_{j_0}$ and $0\leq \frac{3}{2}-\frac{n}{2\alpha}<\f12$, so that
$$
	\frac{R+1}{R^{{\f32-\frac{n}{2\alpha}}}}	 \frac{1}{r_{j_0}^{\frac{n-1}{2}}}\les \frac{1}{r_{j_0}^{\frac{n-3}{2}}}+\frac{1}{r_{j_0}^{\frac{n}{2}-}}.
$$
The spatial integrals converge when $|V(x)|\les \la x\ra^{-\frac{n+1}{2}-}$.

When $\frac{n}{2\alpha}>\frac{3}{2}$,  using $R\gtrsim |t|$  we have
$$
	\frac{1}{|t|^{\f32}} \les \frac{1}{|t|^{\frac{n}{2\alpha}}} R^{\frac{n}{2\alpha}-\f32}.
$$
Furthermore, since $\alpha \geq \frac{n+1}{4}\geq 1$, we have 
$$
	1-\alpha-(k+1)(2\alpha-\frac{n+1}{2})\leq 0, 
$$
so 
we may bound by (with $r_{j_0}=\max(r_j)\approx R$)
\begin{multline*}
	\frac{1}{|t|^{\f32}} \int_{\R^{2k}}	(R+1)\lambda_0^{1-\alpha-(k+1)(2\alpha-\frac{n+1}{2})}\prod_{j=1}^{k+1} \frac{1}{ r_j^{\frac{n-1}{2} }}\, d\lambda \prod_{i=1}^k |V(x_k)|\, d\vec x\\
	\les \frac{1}{|t|^{\frac{n}{2\alpha}}}\int_{\R^{2k}} (R^{\frac{n}{2\alpha}-\frac12}+R^{\frac{n}{2\alpha}-\frac32}) \prod_{j=1}^{k+1} \frac{1}{ r_j^{\frac{n-1}{2} }}\, d\lambda \prod_{i=1}^k |V(x_k)|\, d\vec x\\
	\les \frac{1}{|t|^{\frac{n}{2\alpha}}}\int_{\R^{2k}}\bigg(
	\frac{1}{r_{j_0}^{\frac{n}{2}-\frac{n}{2\alpha}}}+\frac{1}{r_{j_0}^{\frac{n+2}{2}-\frac{n}{2\alpha}}}\bigg)
	\prod_{j=1, j\neq j_0}^{k+1} \frac{1}{ r_j^{\frac{n-1}{2} }}\, d\lambda \prod_{i=1}^k |V(x_k)|\, d\vec x,
\end{multline*}
noting that $\frac{n}{2}-\frac{n}{2\alpha}\geq 0$ and $\frac{n+2}{2}-\frac{n}{2\alpha}<\frac{n}2$,
the spatial integrals are controlled as before, provided $|V(x)|\les \la x\ra^{-\beta}$ for some $\beta>\frac{n+1}{2}$.

\end{proof}

The following lemma takes care of the contribution of tail of the Born series and finishes the proof of the high energy portion of Theorem~\ref{thm:main n}.

	\begin{lemma}\label{lem:tail hi hi d}
	
	Fix $\frac{n+1}4 \leq \alpha<\frac{n}2$ and $K$ sufficiently large.  If $H$ has no embedded eigenvalues and $|V(x)|\les \la x\ra^{-\beta}$ for some $\beta>\frac{n}{2\alpha}+\f52$,  then
	$$
	\sup_{L\geq 1}\sup_{x,y\in\R^n}\bigg|\int_0^\infty e^{it\lambda^{2\alpha}} \lambda^{2\alpha-1}\widetilde \chi(\lambda) \chi(\lambda/L) [(\mR_0V)^K\mR_V(V\mR_0)^K](\lambda^{2\alpha})(x,y)\, d\lambda \bigg| \les \frac{1}{|t|^{\frac{n}{2\alpha}}}.
	$$
	
\end{lemma}

\begin{proof}
	
	Let
	$$
	a_{x,y}(\lambda)=\lambda^{2\alpha-1}\widetilde \chi(\lambda) \chi(\lambda/L) e^{-i\lambda |x|}e^{-i\lambda |y|} [(\mR_0V)^K\mR_V(V\mR_0)^K](\lambda^{2\alpha})(x,y).
	$$
	We prove below that for sufficiently large $K$,
	\be\label{eqn:a bds nbig}
	|\partial_\lambda^j a_{x,y}(\lambda)|\les \lambda^{-2} \la x\ra^{\f{1}2-\frac{n}{2\alpha}} \la y\ra^{\f{1}2-\frac{n}{2\alpha}}, \qquad j=0,1,2.
	\ee
	Using these bounds, it suffices to consider
	$$
	\int_1^\infty e^{it\lambda^{2\alpha}+i\lambda(|x|+|y|)} a_{x,y}(\lambda)\, d\lambda.
	$$
	It is clear that this integral is bounded by one uniformly in $x,y$.  For the time decay, as in the proof of Lemma~\ref{lem:Born hi}, we consider the cases of $\lambda \sim \lambda_0=(-\frac{|x|+|y|}{2\alpha t})^{\frac1{2\alpha-1}}$ and $\lambda \not \sim \lambda_0$.
	
	In the case when $\lambda \not \sim \lambda_0$, two integrations by parts and \eqref{eqn:a bds nbig} results in
	$$
	\bigg|\int_1^\infty e^{it\lambda^{2\alpha}+i\lambda(|x|+|y|)} a_{x,y}(\lambda)\, d\lambda\bigg|\les \frac{1}{|t|^2} \int_1^\infty \lambda^{-2}\, d\lambda \les \frac{1}{|t|^2}.
	$$
	When $\lambda \sim \lambda_0\gtrsim 1$, by Van der Corput and \eqref{eqn:a bds nbig}, we have
	\begin{multline*}
	\bigg|\int_1^\infty 	e^{it\lambda^{2\alpha}+i\lambda(|x|+|y|)} a_{x,y}(\lambda)\, d\lambda\bigg|\les |t|^{-\f12} \la x\ra^{\f{1}2-\frac{n}{2\alpha}}\la y\ra^{\f{1}2-\frac{n}{2\alpha}} |\lambda_0|^{1-\alpha} \int_{\lambda \sim \lambda_0} \lambda^{-2}\, d\lambda\\
	\les \frac{\la x\ra^{\f{1}2-\frac{n}{2\alpha}}\la y\ra^{\f{1}2-\frac{n}{2\alpha}}}{|t|^{\f12}}\les |t|^{-\frac{n}{2\alpha}}.
	\end{multline*}
	Where we used that $\lambda_0\gtrsim 1$, which in particular implies that $|t|\les |x|+|y|$.
	
	To complete the proof, we must establish the bounds in \eqref{eqn:a bds nbig}.	
	Notice that
	\begin{multline}\label{eqn:axy ugly}
		\partial_\lambda^j a_{x,y}(\lambda)=\sum
		\partial_\lambda^{j_1}[\lambda^{2\alpha-1}\widetilde \chi(\lambda) \chi(\lambda/L)]\partial_\lambda^{j_2}[e^{-i\lambda|x|}\mR_0(\lambda^{2\alpha})(x,\cdot)]V\\
		\times\partial_\lambda^{j_3}[(\mR_0V)^{K-1}\mR_V(V\mR_0)^{K-1}](\lambda^{2\alpha})
		V\partial_\lambda^{j_4}[e^{-i\lambda|y|}\mR_0(\lambda^{2\alpha})(\cdot, y)],
	\end{multline}
	where the sum is taken of $j_i\geq 0$ with $\sum j_i=j$.
	
	We note that, since $\lambda\gtrsim 1$, by Proposition~\ref{prop:rep} we have
	$$
	|\partial_\lambda^j e^{-i\lambda |x|}\mR_0^+(\lambda^{2\alpha})(x,x_1)V(x_1)|\les \frac{\la x_1\ra^{j-\beta} }{|x-x_1|^{\frac{n-1}{2}}}
	$$
	We note that $\frac{n-1}{2}\geq \frac{n}{2\alpha}-\frac12$.
	So that,   we have 
	$$
	\|\partial_\lambda^j e^{-i\lambda |x|}\mR_0^+(\lambda^{2\alpha})(x,\cdot)V(\cdot)\la \cdot\ra^{\f52-j+}\|_2\les
	\| \la \cdot\ra ^{\frac{5}{2}-\beta+}|x-\cdot|^{\frac{1-n}{2}}\|_2
	\les \la x\ra^{\frac12-\frac{n}{2\alpha}}
	$$
	provided that $|V(x)|\les \la x\ra^{-\beta}$ for some
	$\beta>\frac{n}{2\alpha}+\frac{5}{2}$.  
	The remainder of the proof mimics that of Lemma~\ref{lem:tail hi}.  By iterating the limiting absorption principle in Proposition~\ref{prop:LAP} sufficiently, the bounds follow since $1-2\alpha<0$.  
	
	The decay on $V$ is necessitated by when all derivatives act on a single resolvent.  If this resolvent is an inner resolvent, to apply Proposition~\ref{prop:LAP} we need multiplication by $V$ to map $L^{2,-\f52-}\to L^{2,\f12+}$, which necessitates $\beta>3$.  If all derivatives in \eqref{eqn:axy ugly} act on the first and second (respectively last and second to last) resolvents, we need to bound
	$$
	\|\partial_\lambda^{j_1} e^{-i\lambda |x|}\mR_0^+(\lambda^{2\alpha})(x,\cdot)V(\cdot)\la \cdot\ra^{\f52-j_1+}\|_2 \|\partial_\lambda^{2-j_1} \mR_0(\lambda^{2\alpha})\|_{L^{2,\f52-j_1+}\to L^{2,j_1-\f52-}},
	$$
	This requires $\beta>\frac{n}{2\alpha}+\frac52$ as described above.

\end{proof}

	\subsection{Low energies}
	
	We now consider the low energy, when $0<\lambda<\lambda_0$ for a sufficiently small constant $\lambda_0\ll1$.
	We utilize the symmetric resolvent identity,
	\eqref{eqn:sym res}.  The low energy claim of Theorem~\ref{thm:main n} follows from 
	\begin{prop}\label{prop:low tail2}
		
		If zero is a regular point of $H$ and $|V(x)|\les \la x\ra^{-\beta}$ for some $\beta>n+4$, then
		$$
		\sup_{x,y\in\R^n}\bigg|\int_0^\infty e^{it\lambda^{2\alpha}} \lambda^{2\alpha-1} \chi(\lambda)  [\mR_V^+-\mR_V^-](\lambda^{2\alpha})(x,y)\, d\lambda \bigg| \les |t|^{-\frac{n}{2\alpha}}.
		$$
		
	\end{prop}
	
	We first establish bounds on $M^{-1}_{\pm}(\lambda)$ and its derivatives.
	
	\begin{lemma}\label{lem:Minv n}
		
		For sufficiently small $\lambda_0$, if $|V(x)|\les \la x\ra^{-\beta}$ for some $\beta>n$ and zero is regular, then the operators $M^\pm(\lambda)$ are invertible on $L^2$.  Furthermore,
		$$
		\bigg\|\sup_{0<\lambda<\lambda_0} |M_\pm^{-1}(\lambda)|+\lambda |\partial_\lambda M_\pm^{-1}(\lambda)|+\lambda^2 |\partial_\lambda^2 M_\pm^{-1}(\lambda)|\,\bigg\|_{L^2\to L^2}<\infty.
		$$
		In addition, if $\beta>n+4$,
		$$
		\bigg\|\sup_{0<\lambda<\lambda_0} \sum_{k=0}^2\lambda^{2\alpha-n+k} |\partial_\lambda^k [M_+^{-1}-M_-^{-1}](\lambda)|\,\bigg\|_{L^2\to L^2}<\infty.
		$$
		
	\end{lemma}

	\begin{proof}
		
		From Proposition~\ref{prop:R0 exp}, we write
		$$
		M^+(\lambda)=U+v\mR_0^+(0)v+\mathcal E(\lambda), \qquad \mathcal E(\lambda)=O(\lambda^{n-2\alpha} |v|(x)|v|(y)).
		$$
		If $|V(x)|\les \la x\ra^{-\beta}$ for $\beta>n$, we have
		$$
		\sup_{0<\lambda<\lambda_0} |\mathcal E(\lambda)|\les \lambda_0^{n-2\alpha}\la x\ra^{-\frac{n}2-}\la y\ra^{-\frac{n}2-},
		$$
		which is bounded on $L^2$.
		$T_0^{-1}$ is absolutely bounded by the same argument in Lemma~\ref{lem:Minv 2}.
		
		So, the claim follows by a Neumann series expansion  for sufficiently small $\lambda_0$.  For the derivative, we use the resolvent identity to write
		$$
		\partial_\lambda M_+^{-1}(\lambda)= M_+^{-1}(\lambda)v\partial_\lambda \mR_0^+(\lambda^{2\alpha})v M_+^{-1}(\lambda).
		$$
		We note that, by Proposition~\ref{prop:R0 exp} we have (for $k\leq 2$)
		\begin{multline*}
			\lambda^{k}|\partial_\lambda^k \mR_0^+(\lambda^{2\alpha})(x,y)|\les (1+\lambda|x-y|)^k\frac{\la \lambda |x-y|\ra^{\f{n+1}2-2\alpha}}{|x-y|^{n-2\alpha}}
			\les \frac{1}{|x-y|^{n-2\alpha}}+|x-y|^{\f{5-n}2}.
		\end{multline*}
		The last bound is seen by considering the cases of $\lambda|x-y|<1$ and $\lambda |x-y|\geq 1$ separately.  Then, for $k\leq 2$
		$$
		\sup_{0<\lambda<\lambda_0} \lambda^{k}|\partial_\lambda^k[ v\mR_0^+(\lambda^{2\alpha})v](x,y)|
		$$
		is bounded on $L^2$ provided $\beta>5$.
		
		The second claim follows from the resolvent identity:
		\begin{multline*}
			[M_+^{-1}-M_{-}^{-1}](\lambda)=-M_{+}^{-1}(\lambda)v[\mR_0^+(\lambda^{2\alpha})-\mR_0^-(\lambda^{2\alpha})]vM_{-}^{-1}(\lambda)\\
			=-\lambda^{n-2\alpha}M_{+}^{-1}(\lambda)v[e^{i\lambda r}F_+(\lambda r)+e^{-i\lambda r} F_{-}(\lambda r)]vM_{-}^{-1}(\lambda)
		\end{multline*}
		The claim now follows as above from \eqref{Fbounds2} and the bounds on $M_{\pm}^{-1}$.
		
	\end{proof}
	
	We are now ready to proof Proposition~\ref{prop:low tail2} and hence Theorem~\ref{thm:main n}.
	
	\begin{proof}[Proof of Proposition~\ref{prop:low tail2}]

	As before, using the symmetric resolvent identity \eqref{eqn:sym res} it suffices to control the contribution of $\mR_0^-(\lambda^{2\alpha})v M_+^{-1}(\lambda) v[\mR_0^+(\lambda^{2\alpha})-\mR_0^-(\lambda^{2\alpha})]$ and $\mR_0^-(\lambda^{2\alpha})v [M_+^{-1}-M_-^{-1}](\lambda) v\mR_0^+(\lambda^{2\alpha})$ to the Stone's formula.    In both cases we consider an operator $\Gamma(\lambda)$ of the form where
	$$
	\widetilde \Gamma:=\sup_{0<\lambda<\lambda_0}( |\Gamma(\lambda)|+\lambda |\partial_\lambda\Gamma(\lambda)|+\lambda^{2} |\partial_\lambda^2\Gamma(\lambda)|)
	$$
	is bounded on $L^2$.  By Lemma~\ref{lem:Minv n} both $M_+^{-1}(\lambda)$ and $\lambda^{2\alpha-n }[M_+^{-1}-M_-^{-1}](\lambda)$ satisfy this bound.
	
	By Proposition~\ref{prop:R0 exp} and the definition of $\Gamma(\lambda)$ above, we need to control
	\be\label{eqn:low tail goal hi n}
	\int_0^1  e^{it\lambda^{2\alpha}+i\lambda (|x|\mp |y|)} \lambda^{n-1}\chi(\lambda)  a_{x,y}(\lambda) \, d\lambda,
	\ee
	where (with $r_1=|x-z_1|$ and $r_2=|z_2-y|$)
	\be
	a_{x,y}(\lambda)=\chi(\lambda)\int_{\R^{2n}}  e^{i\lambda (r_1- |x|\pm (r_2-|y|))} \frac{F(\lambda r_1)}{r_1^{n-2\alpha}} [v\Gamma(\lambda)v](z_1,z_2)
	\bigg(\frac{F(\lambda r_2)}{ r_2^{n-2\alpha}}+F_{\pm}(\lambda r_2) \bigg)\, dz_1 \, dz_2.
	\ee
	Note that, using $|F(\cdot)|,|F_\pm(\cdot)|\les 1$, we have
	$$
	|a_{x,y}(\lambda)|\les \bigg\| \frac{v(\cdot)}{|x-\cdot|^{n-2\alpha}}\bigg\|_{L^2} \| \widetilde \Gamma \|_{L^2\to L^2}\bigg\| v(\cdot)\bigg(1+\frac{1}{|\cdot-y|^{n-2\alpha}}\bigg)\bigg\|_{L^2}\les 1,
	$$
	uniformly in $x,y\in \R^n$.  This suffices for the case of $|t|\les 1$.  Further, by Proposition~\ref{prop:rep}, Lemma~\ref{lem:Minv n}, and the definition of $\widetilde\Gamma$ above, we have 
	\be\label{eqn:ax big n}
		\big|\partial_\lambda^j a_{x,y}(\lambda)\big|\les \lambda^{-j}, \qquad j=0,1,2.
	\ee	
	We now turn to the large time decay, when $|t|\gg 1$ we break up the $\lambda$ integral into two pieces.  When $0<\lambda <|t|^{-\f1{2\alpha}}$ we bound by
	$$
		\int_0^{|t|^{-\f1{2\alpha}}} \lambda^{n-1}\, d\lambda \les |t|^{-\f{n}{2\alpha}}.
	$$
	We now consider the remaining portion, first when $|\lambda (|x|\mp |y|)|\les 1$.  By \eqref{eqn:ax big n} and the assumption that $|\lambda (|x|\mp |y|)|\les 1$, we have
	$$
		\big|\partial_\lambda^j \big[e^{i\lambda (|x|\mp |y|)}a_{x,y}(\lambda)\big]\big|\les \lambda^{-j}, \qquad j=0,1,2.
	$$
	We integrate by parts against $e^{it\lambda^{2\alpha}}$ twice to bound by 
	\be\label{eqn:ibp hi d low}
		\frac{1}{|t|^2}\int_{|t|^{-\frac{1}{2\alpha}}}^1 \lambda^{n-1-4\alpha }\, d\lambda\les |t|^{-\frac{n}{2\alpha}},
	\ee
	where we used that $n-1-4\alpha<-1$. 
	
	Now we consider when $|\lambda(|x|\mp|y|)|\gg 1$, where the phase has a critical point at $\lambda_0=(\frac{-(|x|\mp |y|)}{2\alpha t})^{\frac{1}{2\alpha-1}}$. We first consider when $\lambda \not \sim \lambda_0$.  Here we integrate by parts twice using  that $|\partial_\lambda^j a_{x,y}(\lambda)|\les \lambda^{-j}$, we may bound by
	$$
		\frac{1}{|t|^2}\int_{|t|^{-\frac{1}{2\alpha}}}^1 \bigg| \partial_\lambda \bigg(\frac{1}{\phi'(\lambda)} \partial_\lambda\bigg( \frac{\lambda^{n-1}}{\phi'(\lambda)} a_{x,y}(\lambda)\bigg)\bigg) \bigg|\, d\lambda.
	$$
	Since $\lambda \not \sim \lambda_0$ we have that $|\phi'(\lambda)|\gtrsim \lambda^{2\alpha-1}$ and this is dominated by
	\eqref{eqn:ibp hi d low}.
	
	When $\lambda \sim \lambda_0\les 1$, we have $|t|\gtrsim |\,|x|\mp |y|\,|$.  We integrate by parts once against $e^{it\lambda^{2\alpha}}$ to bound by (and denoting $R:=|x|\mp|y|$)
	\begin{multline*} \label{eqn:low tail hi d}
		\frac{1}{2\alpha it} \int_{\lambda \sim \lambda_0} e^{it\lambda^{2\alpha}} \partial_\lambda \bigg(\lambda^{n-2\alpha} e^{i\lambda R}a_{x,y}(\lambda) \bigg)\, d\lambda\\
		=\frac{1}{2\alpha it} \int_{\lambda \sim \lambda_0}  e^{it\phi(\lambda)} \partial_\lambda \bigg(\lambda^{n-2\alpha} a_{x,y}(\lambda) \bigg)\, d\lambda
		+\frac{iR}{2\alpha it} \int_{\lambda \sim \lambda_0}  e^{it\phi(\lambda)}  \lambda^{n-2\alpha} a_{x,y}(\lambda) \, d\lambda
	\end{multline*}
	We will use that (considering cases of $|x|<|y|$ and $|y|<|x|$)
	\be\label{eqn:ax low big n}
	|\partial_\lambda^j a_{x,y}(\lambda)|\les \lambda^{-j}(\lambda \la x\ra +\lambda\la y\ra)^{\f12-\frac{n}{2\alpha}}
	\ee
	to apply Van der Corput to bound by
	\begin{multline*}
		\frac{\lambda_0^{1-\alpha}}{|t|^{\f32}}\bigg(
		\int_{\lambda \sim \lambda_0}|\partial_\lambda^2 (\lambda^{n-2\alpha} a_{x,y}(\lambda) )|\, d\lambda
		+|R|
		\int_{\lambda \sim \lambda_0}  |\partial_\lambda (\lambda^{n-2\alpha} a_{x,y}(\lambda) |\, d\lambda
		\bigg)\\
		\les (\lambda_0 \la x\ra +\lambda_0\la y\ra)^{\f12-\frac{n}{2\alpha}}\bigg( \frac{\lambda_0^{1-\alpha}}{|t|^{\f32}} \lambda_0^{n-1-2\alpha}+\frac{|R|\lambda_0^{1-\alpha}}{|t|^{\f32}} \lambda_0^{n-2\alpha}\bigg)
		\les (\la x\ra +\la y\ra)^{\f12-\frac{n}{2\alpha}} \frac{|R|\lambda_0^{(\frac{n}{2\alpha}-\f32)(2\alpha-1)}}{|t|^{\f32}},
	\end{multline*}	
	using that $\lambda |R|\gg1$ in the last inequality.  Using that $\lambda_0^{2\alpha-1}\sim \frac{|R|}{|t|}$, we may bound by
	$$
		(\la x\ra +\la y\ra)^{\f12-\frac{n}{2\alpha}} \frac{|R|^{(\frac{n}{2\alpha}-\f12)}}{|t|^{\f{n}{2\alpha}}}
		\les |t|^{-\f{n}{2\alpha}},
	$$
	uniformly in $x,y$.
	
	The bounds in \eqref{eqn:ax low big n} follow from Proposition~\ref{prop:rep}, Lemma~\ref{lem:Minv n}, noting that $\frac{n}{2\alpha}-\f12\leq \frac{n-1}{2}$.

	\end{proof}

\section{Spectral characterization}\label{sec:spec}

We have defined zero energy to be regular if the operator $T_0=U+vG_0v$ is invertible on $L^2$ where $G_0$ is the kernel of $[(-\Delta)^{\alpha}]^{-1}$, which is a multiple of $|x-y|^{2\alpha-n}$.  Further, $S_1$ is the projection onto the kernel of $T_0$.  We note that $S_1$ is finite rank since $vG_0v$ is compact.  As usual, we wish to relate the concept of regularity at zero to distributional solutions to $[(-\Delta)^\alpha+V]\psi=0$. 

We must consider cases based on the relative sizes of $\alpha$ and $n$.  
We consider small $n$ first.

\begin{lemma}\label{lem:spec small a}
	
	In dimensions $n\geq2$, fix $\frac{n}{4}\leq \alpha<\frac{n}2$ and assume that $|V(x)|\les \la x\ra^{-\beta}$ for some $\beta>2\alpha$.  Then $\phi\in S_1L^2(\R^n)$ if and only if $\phi=Uv\psi$ for some $\psi\in L^{2,-\alpha-}(\R^n)$ with $[(-\Delta)^\alpha+V]\psi=0$ in the sense of distributions.  Furthermore, $\psi\in L^\infty(\R^n)$.
	
\end{lemma}

\begin{proof}
	
	We first assume that $\phi\in S_1L^2(\R^n)$, then using $U^2=I$ we have
	$$
		(U+vG_0v)\phi=0, \qquad \Rightarrow \qquad \phi=-UvG_0v\phi. 
	$$	
	We define $\psi=-G_0v\phi$, which implies  $\phi=Uv\psi$.  In particular, this shows that $\psi=-G_0V\psi$ hence $[I+G_0V]\psi=0$, which is equivalent to $[(-\Delta)^\alpha+V]\psi=0$ in the sense of distributions. 
	We now show that $\psi\in L^{2,-\alpha-}$.  By definition we have
	$$
		\psi(x)=-G_0v\phi(x)=-c_{n,\alpha} \int_{\R^n} \frac{v(y)\phi(y)}{|x-y|^{n-\alpha}}\, dy.
	$$
	In particular, $G_0$ is a scalar multiple of the fractional integral operator $I_{2\alpha}$.  By Lemma~2.3 in \cite{Jen}, we have that $I_{2\alpha}:L^{2,s}\to L^{2,-s'}$ provided that $s,s'>2\alpha-\frac{n}2$ and $s+s'>2\alpha$.  Since $\phi\in L^2$, we have that $v\phi \in L^{2,\frac{\beta}2}$, since $\alpha<\frac{n}2$ we have $2\alpha-\frac{n}2<\alpha$, and hence $\frac{\beta}{2}>2\alpha-\frac{n}{2}$. This implies $G_0v\phi\in L^{2,-s'}$ for some $s'>2\alpha-\frac{n}{2}$, and since $\alpha<\frac{n}2$ we may select $s'=\alpha+$ and conclude that $\psi\in L^{2,-\alpha-}$ as desired provided that $\beta>2\alpha$.
	
	Now, assuming that $\psi\in L^{2,-\alpha-}$ satisfies $[(-\Delta)^\alpha+V]\psi=0$ in the sense of distributions, and let $\phi=Uv\psi$.  Then,
	$$
		(U+vG_0v)\phi=v\psi+vG_0V\psi=v[I+G_0V]\psi=0.
	$$
	Hence $\phi\in S_1L^2(\R^n)$.
	
	Now, to show that $\psi\in L^\infty$ we first consider when $\frac{n}{4}<\alpha<\frac{n}2$, then $0<n-2\alpha < \frac{n}{2}$, so that $G_0(x,y)$ is a locally $L^2$ function of $y$ uniformly in $x$.  Further, under the decay conditions we have
	$$
		|\psi(x)|\les \int_{\R^n}|G_0(x,y)V(y)|\, |\psi(y)| \, dy
		\les \|G_0(x,\cdot)V(\cdot)\|_2 \|\psi\|_2
	$$
	uniformly in $y$, hence $\psi\in L^\infty$ as claimed.
	
	When $\alpha=\frac{n}4$, one needs to iterate the resolvent identity once since $n-2\alpha=\frac{n}2$ in this case.  We write $\psi(x)=G_0VG_0V\psi(x)$ and use Lemma 6.3 in \cite{EG1} to see that
	$$
	|G_0VG_0(x,y)|\les \int_{\R^n} \frac{\la z\ra^{-\beta}}{|x-z|^{n-2\alpha}|z-y|^{n-2\alpha}}\, dz
	\les \la x-y \ra^{2\alpha-n}(|x-y|^{0+}+|x-y|^{0-})
	$$
	since $n+\beta-4\alpha>n-2\alpha$.  It is clear that $(G_0V)^2(x,y)$ is in $L^2$ uniformly in $x$ and the claim follows again by Cauchy-Schwartz.
	
\end{proof}

Heuristically, this allows for the possibility of zero energy resonances when $2\alpha<n\leq 4\alpha$.  When $n>4\alpha$, we expect that no threshold resonances may exist.  For decaying potentials, this is a consequence of the following.

\begin{lemma}\label{lem:spec big a}
	
	In dimensions $n>2$, fix $0< \alpha<\frac{n}4$ and assume that $|V(x)|\les \la x\ra^{-\beta}$ for some $\beta>4\alpha$.  Then $\phi\in S_1L^2(\R^n)$ if and only if $\phi=Uv\psi$ for some $\psi\in L^{2}(\R^n)$ with $[(-\Delta)^\alpha+V]\psi=0$ in the sense of distributions.  Furthermore, $\psi\in L^\infty(\R^n)$.
	
\end{lemma}

\begin{proof}
	
	As in the previous proof, if $\phi \in S_1L^2(\R^n)$ then $\psi=-G_0v\phi$ is a distributional solution of $H\psi=0$.  Since $n>4\alpha$, by Lemma~2.3 in \cite{Jen} the operator $I_{2\alpha}:L^{2,s}\to L^2$ provided $s>2\alpha$.  Hence, $\psi=-G_0v\phi\in L^2$ provided $\beta>4\alpha$.
	
	The claim that $\psi\in L^\infty$ follows by iterating the identity $\psi=-G_0V\psi$ to write $\psi=[-G_0V]^k\psi$ for a sufficiently large $k$ depending on $n$ and $\alpha$.  Then, using Lemma 6.3 in \cite{EG1} one has
	$$
	|G_0VG_0(x,y)|\les \int_{\R^n} \frac{\la z\ra^{-\beta}}{|x-z|^{n-2\alpha}|z-y|^{n-2\alpha}}\, dz
	\les \left\{ 
	\begin{array}{ll}
		(\frac{1}{|x-y|})^{\max(0,n-4\alpha)} & |x-y|\leq 1\\
		(\frac{1}{|x-y|})^{\min(n-2\alpha,n+\beta-4\alpha)} & |x-y|>1	
	\end{array}	\right.
	$$
	As before, $n+\beta-4\alpha>n-2\alpha$, so one lessens the local singularity by a factor of $2\alpha$ without diminishing the decay for large $|x-y|$ after applying $G_0V$.  More succinctly,
	$$
	|G_0VG_0(x,y)|\les \int_{\R^n} \frac{\la z\ra^{-\beta}}{|x-z|^{n-2\alpha}|z-y|^{n-2\alpha}}\, dz
	\les 
	\frac{\la x-y\ra^{2\alpha-n+ \max(0,n-4\alpha)} }{|x-y|^{\max(0,n-4\alpha)}}
	$$
	One can iterate by noting that $\la z_1\ra^{-\beta}\la z_1-y\ra^{-\gamma}\les \la z_1\ra ^{-(\beta+\gamma)}+\la z_1-y\ra ^{-(\beta+\gamma)}$ so that
	\begin{multline*}
		|(G_0V)^2G_0(x,y)|\les \int_{\R^n} \frac{\la z_1\ra^{-\beta}\la z_1-y\ra^{{2\alpha-n+ \max(0,n-4\alpha)}}}{|x-z_1|^{n-2\alpha} |z_1-y|^{\max(0,n-4\alpha)} } \,dz_1 \\
		\les   	\int_{\R^n} \frac{\la z_1\ra^{-\beta+2\alpha-n+ \max(0,n-4\alpha)}}{|x-z_1|^{n-2\alpha} |z_1-y|^{\max(0,n-4\alpha)} } \,dz_1+ \int_{\R^n} \frac{\la z_1-y\ra^{-\beta+2\alpha-n+ \max(0,n-4\alpha)}}{|x-z_1|^{n-2\alpha} |z_1-y|^{\max(0,n-4\alpha)} } \,dz_1\\
		\les \frac{ \la x-y\ra^{2\alpha-n+\max(0,n-6\alpha)} }{|x-y|^{\max(0,n-6\alpha)}}.
	\end{multline*}
	Where we apply Lemma 6.3 of \cite{EG1}  on each integral, using the change of variables $z_1\mapsto z_1+y$ in the second integral.
	Iterating $k$ times, we arrive at
	$$
	|(G_0V)^kG_0(x,y)|\les \frac{ \la x-y\ra^{2\alpha-n+\max(0,n-(j+1)\alpha)} }{|x-y|^{\max(0,n-(k+1)\alpha)}}.
	$$
	For $k\geq\lceil \frac{n}{2\alpha}\rceil$, one sees that  $(G_0)^kG_0(x,y)$ is an $L^2$ function of $y$ uniformly in $x$. Hence for $k>\lceil \frac{n}{2\alpha}\rceil$
	$$
	|\psi(x)|=|(G_0V)^kG_0\psi(x)|\les \int_{\R^n} |(G_0V)^kG_0(x,y)|\, |\psi(y)|\, dy\les \|(G_0V)^kG_0(x,\cdot)\|_2 \|\psi\|_2,
	$$
	proving the claim.
	
\end{proof}

\end{document}